\documentclass{article}

\usepackage[english]{babel}

\usepackage[letterpaper,top=2cm,bottom=2cm,left=3cm,right=3cm,marginparwidth=1.75cm]{geometry}

\usepackage{amsmath,amssymb,amsthm}
\usepackage{graphicx}
\usepackage{tikz}
\usepackage{circuitikz}
\usetikzlibrary{arrows,positioning}
\usepackage{caption}
\usepackage{subcaption}
\usepackage{xcolor}
\usepackage[colorlinks=true, allcolors=blue]{hyperref}

\DeclareMathOperator{\ima}{im}

\DeclareMathOperator{\rk}{rk}

\DeclareMathOperator{\re}{Re}
\newcommand{\dR}{\mathbb{R}}

\newcommand{\dC}{\mathbb{C}}
\newcommand{\dN}{\mathbb{N}}
\newcommand{\K}{\mathbb{K}}
\newcommand{\Hc}{\mathcal{H}}

\newcommand{\Ac}{\mathcal{A}}
\newcommand{\Bc}{\mathcal{B}}
\newcommand{\Cc}{\mathcal{C}}
\newcommand{\Dc}{\mathcal{D}}

\newcommand{\Tc}{\mathcal{T}}

\usepackage{calc}

\makeatletter
\newlength\@SizeOfCirc%
\newcommand{\CricArrowRight}[1]{%
    \setlength{\@SizeOfCirc}{\maxof{\widthof{#1}}{\heightof{#1}}}%
    \tikz [x=1.0ex,y=1.0ex,line width=.15ex, draw=black]%
        \draw [->,anchor=center]%
            node (0,0) {#1}%
            (0,1.2\@SizeOfCirc) arc (85:-240:1.2\@SizeOfCirc);%
}%
{\left(\begin{smallmatrix}}%            begin code
{\end{smallmatrix}\right)}%             end code

\newenvironment{smallbmatrix}%          environment name
{\left[\begin{smallmatrix}}%            begin code
{\end{smallmatrix}\right]}%             end code

\usepackage{tikz-cd}
\usetikzlibrary{positioning,arrows,decorations.markings}
\tikzset{
block/.style={
  draw, 
  rectangle, 
  minimum height=1.5cm, 
  minimum width=2.5cm, align=center,
  fill=blue!20
  }, 
line/.style={->,>=latex'}
}
\tikzset{negated/.style={
      decoration={markings,
           mark= at position 0.5 with {
               \node[transform shape] (tempnode) {$\backslash\!\!\backslash$};
            }
       },
       postaction={decorate}
    }
}

\newtheorem{definition}{Definition}[section]

\newtheorem{proposition}[definition]{Proposition}
\newtheorem{lemma}[definition]{Lemma}
\newtheorem{corollary}[definition]{Corollary}
\newtheorem{assumption}[definition]{Assumption}
\newtheorem{remark}[definition]{Remark}

\theoremstyle{definition}

\newtheorem{example}[definition]{Example}

\title{On discrete-time dissipative port-Hamiltonian (descriptor) systems}

\author{Karim Cherifi\thanks{Institut f\"{u}r Mathematik, Technische Universität Berlin, Stra\ss e des 17.\ Juni 136, 10623 Berlin, Germany  (\texttt{cherifi@math.tu-berlin.de}).} \and Hannes Gernandt\thanks{Fraunhofer IEG, Fraunhofer Research Institution for Energy Infrastructures and Geothermal Systems IEG, Cottbus, Gulbener Straße 23, 03046 Cottbus, Germany  (\texttt{hannes.gernandt@ieg.fraunhofer.de}).}
\and Dorothea Hinsen\thanks{ Institut f\"{u}r Mathematik, Technische Universität Berlin, Stra\ss e des 17.\ Juni 136, 10623 Berlin, Germany   (\texttt{hinsen@math.tu-berlin.de}).} \and Volker Mehrmann \thanks{Institut f\"{u}r Mathematik, Technische Universität Berlin, Stra\ss e des 17.\ Juni 136, 10623 Berlin, Germany   (\texttt{hinsen@math.tu-berlin.de}).} }

\begin{document}
\maketitle

\begin{abstract}
Port-Hamiltonian (pH) systems have been studied extensively for linear continuous-time dynamical systems. This manuscript presents a discrete-time pH descriptor formulation for linear, completely causal, scattering passive dynamical systems based on the system coefficients. The relation of this formulation to positive and bounded real systems and the characterization via positive semidefinite solutions of Kalman-Yakubovich-Popov inequalities is also studied.  
\end{abstract}

\section{Introduction}
System modeling with linear continuous-time  descriptor systems of the form
\begin{align}
\label{cont_DAE}
E\dot x&=Ax+Bu,\quad y=Cx+Du,\quad t\geq 0,\quad x(0)=x^0,
\end{align}
where $x\in C^1(\dR_+,\K^n)$, $u,y\in C(\dR_+,\K^m)$, $E,A\in\K^{n\times n}$, $B\in\K^{n\times m}$, $C\in\K^{m\times n}$ and possibly singular $E$ has been extensively studied, see e.g. \cite{Dai89,KunM06,Meh91}.
Here $\K$ is the field of real $\dR$ or complex numbers $\dC$, $\dR_+$  are the nonnegative real numbers, $\dC_+$ are the complex numbers with nonnegative real part, and $C^k(\dR_+,\K^n)$ is the  set of $k$-times continuously differentiable functions $\dR_+\to \K^n$, where we drop the superscript for $k=0$.

In a similar way linear discrete-time descriptor systems of the form
\begin{align}
\label{discr_DAE}
Ex_{k+1}&=Ax_k+Bu_k,\quad y_k=Cx_k+Du_k,\quad k\geq 0,\quad x_0=x^0,
\end{align}
with sequences of vectors $x_k\in\K^n$, $u_k,y_k\in\K^m$, $k=0,1,2,\ldots$, $E,A\in\K^{n\times n}$, $B\in\K^{n\times m}$, $C\in\K^{m\times n}$, with possibly singular $E$, are well studied, see e.g.
\cite{BanV19,Bru09,Dai89,Meh91}.

Typical examples of discrete-time systems are Lyontief models in economics \cite{Lue77,LueA77,MerL89} or Leslie models in population dynamics, see \cite{Bru09,Cam79}, sampled data systems \cite{CheF91}, or arising from the discretization of continuous-time systems. The descriptor formulation arises in backward Leslie models and when implicit discretization methods are used for continuous-time systems. 

While for general linear descriptor systems  many system properties, like controllability, reachability, observability, 
%stabilizability, detectability, 
passivity, and stability, can be characterized via linear algebra calculations, in the last $20$ years for continuous-time systems a new model class has received a lot of attention. These are the dissipative port-Hamiltonian (pH) descriptor systems for which many of these system properties are directly encoded in the system structure and which satisfy a lot of other important properties that make them an ideal class for energy-based modeling,  see \cite{BeaMXZ18,MehMW18,MehM19,MehS23,AchAM23,SchM18,SchJ14,Sch13}. 

In some recent papers it has also been analyzed when a general descriptor system is equivalent to a pH descriptor system \cite{BeaMV19,BeaMX15,MehS23} by considering different characterizations of passivity, such as positive realness or the solvability of Kalman-Yakubovich-Popov inequalities, see the detailed analysis in \cite{CheGH22}, where the relationship between these concepts is studied, see also the detailed analysis for scattering and impedance passive standard state-space systems summarized in \cite[p.\ 135]{BroLME20}. We stress that the representation as a pH (descriptor) system is generally not unique and any positive definite solution of the associated Kalman-Yakubovich-Popov inequalities will lead to a different pH representation. 

For discrete-time descriptor systems, the theory is much less developed and not even a clear definition of a dissipative discrete-time pH system is available. Most definitions are based on appropriate discretizations of continuous-time systems. In \cite{KotL19} for standard state-space systems and in \cite{MehM19} for descriptor systems  a definition of discretized pH system is presented by discretizing a continuous-time system using collocation
methods or symplectic integrators.
In \cite{FalH16} a numerical scheme is specially designed to preserve the pH structure and the power balance equation to define a class of discrete-time pH systems. In \cite{MorMMN19,YalSK15}   discrete-time port-controlled Hamiltonian systems are constructed via a discrete gradient structure.

Many more discretization based methods are constructed for purely Hamiltonian models, including discrete gradient methods or geometric integrators, see e.g.\ \cite{HaiLW06} or \cite{LaiA06}. We do not discuss these approaches here because they are mainly concerned with the preservation of the Hamiltonian and do not address the dissipation inequality.

In \cite{Mac22} the synthesis of discrete-time systems is based on the mapping (via state-feedback) of the open-loop error system to a target one in pH form. In \cite{TalCS06} the authors propose a discrete-time pH system formulation based on Dirac structures for positive real systems. In \cite{StrSSF05}, for unconstrained robotics systems, a sampled data construction of discrete-time systems with a Dirac structure is presented.

In \cite{CerSB07} the authors study the composition of Dirac structures in power variables and wave variables (scattering)
representation. They show that the latter case corresponds to the Redheffer star product of unitary mapping.

In \cite{TalCS06} for standard state-space systems a discrete-time pH version is constructed via Dirac structures but this approach has not yet been extended to descriptor systems.

All these publications mainly consider pH systems in the positive real case or for impedance passive systems. However, in the literature, bounded real and corresponding scattering passive systems are also studied. The latter class is
motivated by applications in quantum mechanics, quantum field theory  and digital integrated circuits design. E.g.\ in \cite{BraGZC22}, scattering data from a bounded real transfer function is used in macro-modeling for Electronic Design Automation. This is enabled by passivity preserving realizations that result in a scattering passive system representation.

In this manuscript, we take a different approach and introduce a definition of standard discrete-time pH systems in the scattering passive setting by considering positive definite solutions of Kalman-Yakubovich-Popov (KYP) inequalities for discrete-time systems. The definition is then extended to impedance passive systems via external Cayley transform. In the standard case, i.e. if $E=I$, the pH definition is based only on the coefficient matrices of the state-space realization and hence allows for both discrete-time systems and discretization of a continuous-time system. We then analyze the relationship to other passivity definitions such as positive and bounded realness, and present implication charts that show the subtle differences and extra assumptions that have to be made. As in the continuous-time case, the representation of discrete-time pH descriptor systems is, in general, not unique and different representations can e.g. be obtained via different positive definite solutions of discrete-time KYP inequalities as well as different system transformations. 

In Section \ref{sec:Dissipation}, we study the relation of the obtained representation for semiexplicit completely causal descriptor systems to the concept of passivity and transfer function properties. An overview of the relations between positive real, KYP and passivity for discrete-time standard state-space systems has been given in \cite[Section 3.15]{BroLME20}. 
In Section~\ref{sec:phDef} we propose an extension of the pH framework to discrete-time descriptor systems using the scattering supply rate rather than the impedance supply rate which is typically used to define continuous-time pH systems.
In Section \ref{sec: externalcayley} we discuss the relationship between scattering and impedance passive discrete-time pH representations using the external Cayley transform. This transform was discussed previously in the literature \cite{Sta03} for standard state-space systems and we extend this approach to pH descriptor systems.
In Section \ref{sec:cont_to_discr}, we study passivity preserving discretizations of continuous-time systems and show that they lead
to passive discrete-time systems and eventually to discrete-time pH systems. 

%Finally, the geometric formulation of discrete-time port-Hamiltonian systems and the interconnection of these systems is discussed in Section \ref{sec:Geometric_form}.

\section{Notation and preliminaries} \label{sec:prelim}

We assume throughout the paper that the descriptor system, (continuous- or discrete-time)  as well as the pair of coefficient matrices $(E,A)$ are \emph{regular}, i.e.\ $\det(\lambda E-A)\neq 0$ holds for some $\lambda\in\dC$. The set of all such $\lambda\in\dC$ for which $\lambda E-A$ is invertible is called the \emph{resolvent set} and is denoted by $\rho(E,A)$. The complement $\dC\setminus\rho(E,A)$ is called the \emph{spectrum} and denoted by $\sigma(E,A)$. 
For a matrix $A \in \mathbb{K}^{n,m}$ $A^\top, A^H,A^{-H}$, $\rk(A)$ denote the transpose, conjugate transpose inverse of $A^H$ and rank of $A$, respectively.
To indicate that a (real symmetric) Hermitian matrix $M\in\mathbb{K}^{n,n}$ is positive definite, or positive semidefinite, we write $M>0$ and $M\geq 0$, respectively.

It is well-known, see e.g. \cite{BunBMN99,Gan59b,Meh91}, that for regular pairs $(E,A)$ there exist invertible matrices $S,T\in\K^{n\times n}$, $r\in\dN$, nilpotent $N\in\K^{(n-r)\times(n-r)}$ and $A_f\in\K^{r\times r}$ such that 
\begin{align}
\label{eq:weier}
(SET,SAT)=\left(\begin{bmatrix}I_r&0\\0&N
\end{bmatrix}, \begin{bmatrix} A_f&0\\0&I_{n-r}
\end{bmatrix}\right).
%,\quad SB=\begin{bmatrix}
%B_1\\ B_2
%\end{bmatrix}.
\end{align}
%}
If $S$ and $T$ are chosen in such a way that $A_f$ is in (real) Jordan canonical form, then \eqref{eq:weier} is called the \emph{Weierstra\ss\ form} of $(E,A)$ and the \emph{index} $\nu$ of $(E,A)$ is defined as the nilpotency index of $N$. Thus it follows immediately that the spectrum of a regular pair $(E,A)$ is a finite set.

\subsection{Controllability, observability and minimality}
In the following, some controllability and observability notions for descriptor systems are recalled. 
Let $S_\infty(E)$ be a matrix with columns that span the kernel of $E$ and let $T_\infty(E)$ be a matrix with columns that span the kernel of $E^H$, then we introduce the following controllability and observability conditions for %continuous- and 
discrete-time descriptor systems 
%\eqref{cont_DAE}, %respectively 
\eqref{discr_DAE}, see \cite{BunBMN99}
\begin{itemize}
\setlength{\itemindent}{2em}
% \item [(C0)]\   $\rk [ \alpha E - %\beta A,\, B ] = n$
 %       for all  $(\alpha,\beta)\in %\C^2$.
\item [(C1)]\  $\rk\, [ \lambda E - A,\, B ] = n$  for all  $\lambda \in \dC$, \label{gl:C1}
\item [(C2)]\ $\rk\, [ E,\, A S_\infty(E),\, B ] = n$.
\end{itemize} 
A regular system is called 
%{\em completely controllable or C-controllable}, if
%(C0) holds it is called 
{\em behaviorally controllable or  R-controllable},
if (C1) holds, and it is called 
{\em strongly controllable}
%or  S-controllable},
if (C1) and (C2) hold. \label{gl:str_contrl}
Note that the controllability properties (C1) and (C2) can also be defined equivalently in terms of reachable sets, see \cite{BerR13}. 
%Complete controllability ensures 
%that for any given initial and final %states
%$x_0$, $x_f$, respectively, there &exists a control that transfers the %system from $x_0$ to $x_f$ in finite %time, while 

Strong controllability ensures that for any given consistent initial and final states $x_0$, $x_K$  there exists a control sequence $(u_j)_{j=1,\ldots,k}$ that transfers the system from $x_0$ to $x_K$. 
%Regular systems that satisfy
%Condition~(C2) are called {\em  controllable at %infinity\/} or {\em impulse controllable\/}. 

Similarly, we can define dual observability concepts \cite{Sok06, Dai89, BerRT17}
\begin{itemize}
\setlength{\itemindent}{2em}
 %\item [(O0)]\   $\rk [ \alpha E^H - %\beta A^H,\, C^H ] = n$
%        for all  $(\alpha,\beta)\in %\C^2$.
\item [(O1)]\ $\rk\, [ \lambda E^H - A^H,\, C^H ] = n$  for all  $\lambda \in \dC$, \label{gl:O1}
\item [(O2)]\ $\rk\, [ E^H,\, A^H T_\infty(E),\, C^H ] = n$. 
\end{itemize} 
A regular descriptor system is called {\em behaviorally observable or  R-observable} if condition (O1) holds. %It is called \emph{strongly observable}
 %or S-observable} if conditions (O1) and (O2) hold. 
%A regular system that satisfies condition (O2) is called
%{\em observable at infinity or impulse-observable}. 
It is called \emph{strongly observable
 or S-observable} if conditions (O1) and (O2) hold. \label{gl:str_obsv} Note that equivalent definitions of the observability properties (O1) and (O2) in terms of solutions of the systems are given in \cite{BerRT17}.

For standard discrete-time state-space systems
(i.e. for $E=I$) conditions (C1) and (O1) are the classical controllability and observability conditions.

In the following, we recall some properties of transfer functions of regular discrete-time systems \eqref{discr_DAE} with coefficients $(E,A,B,C,D)$. 
The transfer function can be derived using the $z$-transform and is given by 
\begin{equation}\label{tf}
\Tc(z)=C(zE-A)^{-1}B+D
\end{equation}
for all $z\in\dC$ for which the inverse $(zE-A)^{-1}$ exists. %For 

Note that for descriptor systems with index $\nu\geq 2$, the transfer function may have a linear polynomial part and thus may grow unboundedly for $z\rightarrow\infty$. However, if  $\lim_{\vert z\vert\rightarrow \infty} \Tc(z)$ exists then the transfer function is called \emph{proper}. In this case, there exists a realization of the transfer function as a standard-state space system $(I_n,A,B,C,D)$ where $D=\lim_{\vert z\vert\rightarrow \infty} \Tc(z)$ holds.

Given a matrix-valued function $\mathcal{G}:\Omega\rightarrow\dC^{l\times m}$ on some domain $\Omega\subseteq \dC$, where the entries of $\mathcal{G}(z)$ are rational functions in $z$, then we say that a descriptor system \eqref{discr_DAE} with coefficients $(E,A,B,C,D)$ is a \emph{realization} of $\mathcal{G}$, if 
\[
\mathcal{G}(z)=C(zE-A)^{-1}B+D=\Tc(z),\quad z\in \Omega.
\]
Often one is interested in \emph{minimal} realizations of $\mathcal{G}$, i.e.\ realizations where the state-space dimension $n$ is the smallest possible.

For standard discrete-time state-space systems it is well known that the realization $(I_n,A,B,C,D)$ is minimal if and only if the system is controllable and observable, i.e.\ (C1) and (O1) hold, 
see e.g.\  \cite{Wil71}. % and Appendix~\ref{app:co}.

For descriptor systems, besides (C1) and (O1), 
additional conditions are required to achieve minimality, see \cite[Theorem 6.3]{FreJ04}. There it is shown that minimal systems of index at most one are already standard state-space systems. 
%In particular, minimal  completely causal %systems are standard state-space systems.

Minimality conditions for descriptor systems are discussed in \cite{Sok06}, where 
it is shown 
%a realization is called %conditionally minimal if it fulfills the conditions %
that if conditions (C2) and (O2) hold,
%, see also \cite{Dai89}  %in %Appendix~\ref{app:co}. In this case, 
then the order can be reduced and a deflated minimal realization can be constructed by removing the algebraic part of the pair $(E,A)$ and adding this part to the feedthrough matrix $D$.

%\subsection{Causality of discrete-time descriptor %systems} \label{sec:causality}
\subsection{Causality}
Causality is an important property for discrete-time descriptor systems in practice. 
\begin{definition}\label{def:causa}
Let $(E,A)$ be a regular pair of matrices. 
%For a given inhomogeneity sequence $f_k$, 
The discrete-time differential-algebraic system
\begin{equation}\label{dae_discr}
Ex_{k+1}=Ax_k+f_k
\end{equation}
is called \emph{completely causal} if for all inhomogeneities $(f_k)_{k\geq 0}$ and all  $i\geq 1$, $k\geq 0$ the solution $x_k$ 
%given by \eqref{eq:brüll_formula} 
does not explicitly depend on $f_{k+i}$. 
A discrete-time system of the form \eqref{discr_DAE}  is called \emph{input-output causal} if for all consistent input sequences $(u_k)_{k\geq 0}$, $i\geq 1$ and all $k\geq 0$ the output $y_k=Cx_k$ does not explicitly depend on $u_{k+i}$.  
\end{definition}
%By extension, this definition is also valid for %the special case of linear state space systems %where $f_k=B u_k$. \HGcomment{what is meant here?}

\begin{remark}\label{rem:io-causality}{\rm
Clearly, complete causality implies input-output causality but the converse does not hold except if $B,C$ are square and invertible for system \eqref{discr_DAE} where $f_k=Bu_k$. 
Note that input-output causality is known simply as causality in most of the literature, see e.g. \cite{GriG15,Mel16,OppWN96}.
}
\end{remark}

For discrete-time descriptor systems causality is strongly related to the index $\nu$ of the pair $(E,A)$.
\begin{proposition}
\label{prop:causal}
The discrete-time system \eqref{dae_discr} is completely causal if and only if the index $\nu$ is at most one. The descriptor system~\eqref{discr_DAE} is input-output causal if and only if the (transfer) function $z\mapsto D+C(zE-A)^{-1}B$ is proper.
\end{proposition}
\begin{proof}For completeness, the proof of Proposition~\ref{prop:causal} is given in Section~\ref{sec:solvability}.
\end{proof}

The discussed causality analysis  motivates why we restrict ourselves in the following to discrete-time descriptor systems with index $\nu\leq1$. Then we can employ the singular value decomposition of $E$, see e.g. \cite{GolV96}, which determines unitary matrices $U,V\in\K^{n\times n}$ and $\Sigma_E\in\K^{r\times r}$ diagonal and positive definite  such that  
\begin{align}
\label{eq:semi_expl}
(UEV,UAV)=\left(\begin{bmatrix}
\Sigma_E&0\\0&0
\end{bmatrix},\begin{bmatrix}
A_{11}&A_{12}\\A_{21}&A_{22}
\end{bmatrix}\right),
\end{align}
for some $A_{11}\in\K^{r\times r}$, $A_{12},A_{21}^\top\in\K^{r\times (n-r)}$, and invertible $A_{22}\in\K^{(n-r)\times (n-r)}$. 

Compared to the Weierstra\ss\ canonical form \eqref{eq:weier} this structure can be obtained in a numerically stable way. The transformed discrete-time descriptor system then has the form
\begin{align}\label{semiexplicit}
\begin{split}
\begin{bmatrix}
\Sigma_E&0\\0&0
\end{bmatrix}\begin{bmatrix}x_{k+1}^1\\x_{k+1}^2\end{bmatrix}&=\begin{bmatrix}
A_{11}&A_{12}\\A_{21}&A_{22}
\end{bmatrix}\begin{bmatrix}x_{k}^1\\x_{k}^2\end{bmatrix}+\begin{bmatrix}
B_1\\ B_2
\end{bmatrix}u_k,\\ y_k&=[C_1,C_2]\begin{bmatrix}x_{k}^1\\x_{k}^2\end{bmatrix} +D u_k,\quad \begin{bmatrix}x_{k}^1\\
x_{k}^2\end{bmatrix}=V^Hx_k.
\end{split}
\end{align}
If in this system the first equation is not void, %i.e. %r>0$
%cannot be written with only algebraic equations 
then theoretically it can always be reduced to a standard state-space system by resolving the algebraic equation for $x_k^2$ which leads to
\begin{align*}
\Sigma_E x_{k+1}^1&=(A_{11}-A_{12}A_{22}^{-1}A_{21})x_k^1+(B_1-A_{12}A_{22}^{-1}B_2)u_k,\\
y_k&=(C_1-C_2A_{22}^{-1}A_{21})x_{k}^1 +(D-C_2A_{22}^{-1}B_2) u_k.
\end{align*}
Hence, using the modified state $\widehat x_{k}:=\Sigma_E^{\tfrac{1}{2}}x_k^1$ , this leads to the standard discrete-time state-space system 
\begin{align}
\label{eq:indexone_ODE}
\begin{split}
 \widehat x_{k+1}&=\underbrace{\Sigma_E^{-\tfrac{1}{2}}(A_{11}-A_{12}A_{22}^{-1}A_{21})\Sigma_E^{-\tfrac{1}{2}}}_{=:\mathcal{A}}\widehat x_k+\underbrace{\Sigma_E^{-\tfrac{1}{2}}(B_1-A_{12}A_{22}^{-1}B_2)}_{=:\mathcal{B}}u_k,\\
y_k&=\underbrace{(C_1-C_2A_{22}^{-1}A_{21})\Sigma_E^{-\tfrac{1}{2}}}_{=:\mathcal{C}}\widehat x_{k}+\underbrace{(D-C_2A_{22}^{-1}B_2)}_{=:\mathcal{D}}u_k,
\end{split}
\end{align}
together with an algebraic equation 
\[
x_k^2=-A_{12} A_{22}^{-1} x_k^1-B_2 u_k
\]
that defines the relation between input sequences $(u_k)_{k\geq 0}$ and associated consistent initial conditions. Note, however, that in practice and when using numerical control methods, it is better to work with the formulation \eqref{semiexplicit}, since a potential ill-conditioning of $A_{22}$ would lead to very large errors in the reduced representation.

%\subsection{Basic assumptions}
%Regularity, the index and minimality of discrete-time %descriptor systems}
%\label{sec:regul}
%In this section we will present assumptions on the %
%regularity, minimality and the index of discrete-time 
%descriptor systems.
%., to justify 

The following assumption is used throughout this manuscript.
%\HGcomment{I haven't included the assumption %in all of the results yet}
\begin{assumption}
\label{Assumpt}
    The discrete-time descriptor system \eqref{discr_DAE} with coefficients $(E,A,B,C,D)$ is assumed to be regular, $E\neq 0$, and completely causal, i.e.\ $(E,A)$ has index at most one. 
\end{assumption}
%\begin{equation}\label{Assumpt}
%\fbox{\bf $\begin{matrix}
%\text{The discrete-time descriptor system \eqref{discr_DAE} with coefficients} \\ \text{  $(E,A,B,C,D)$ is assumed to be regular, completely causal,} \\ \text{and has $E\neq 0$.} \end{matrix}$}
%\end{equation}
% 

It is well-known, see \cite{BunBMN99}, that for general descriptor systems under  the presented  controllability and observability assumptions, a general discrete-time descriptor system can be transformed via state or output feedback as well as differentiation of uncontrolled parts to satisfy this assumption. However, these regularizing feedbacks in general do not preserve the passivity of the system and therefore we restrict ourselves to systems that satisfy Assumption~\ref{Assumpt}.

After recalling some basic results on discrete-time descriptor systems, in the next section, we derive characterizations of dissipative systems.

\section{Dissipative systems and their characterizations}
\label{sec:Dissipation}
Motivated by \cite{CheGH22}, we consider in this section the relationship between passivity, solutions of the Kalman-Yakubovich-Popov (KYP) inequalities, and  positive and bounded realness for discrete-time systems. After that,
%Then in the next section, 
we  introduce pH systems for discrete-time systems only in terms of coefficients without the use of transformation or discretization. 

%\HGcomment{Motivated by the continuous-time paper, we consider in this section..., make definition of pH only in terms of the coefficients and not ito transformation or discretization, refer to Chapter 4}
We start by investigating the dissipativity of \eqref{discr_DAE} with respect to a given \emph{supply rate} $s:\K^m\times \K^m\rightarrow \dR$ that is mapping pairs $(u,y)\in\K^m\times\K^m$ of input and output variables of \eqref{discr_DAE} to real numbers. We consider different characterizations of dissipativity along with the relation between the different supply rates.

For linear systems, one typically uses quadratic supply rates of the form
\begin{align}
\label{def:supply}
s(y,u):=\begin{bmatrix}
y\\ u \end{bmatrix}^H\begin{bmatrix} Q & S\\S^H& R\end{bmatrix}\begin{bmatrix}
y\\ u
\end{bmatrix},\quad Q,R,S\in\dC^{m\times m},\quad Q=Q^H,\quad R=R^H.
\end{align}
Of special interest are the 
\emph{impedance supply rate}
given by 
\begin{align}\label{impsupply}
s_{imp}(y,u):= \begin{bmatrix}
y\\ u \end{bmatrix}^H\begin{bmatrix} 0 &  I_m\\ I_m & 0\end{bmatrix}\begin{bmatrix}
y\\ u
\end{bmatrix}= 2\Re(y^Hu).
\end{align}
and the \emph{scattering supply rate} given by
\begin{align}\label{scatsupply}
s_{sca}(y,u):=\begin{bmatrix}
y\\ u \end{bmatrix}^H\begin{bmatrix} -I_m & 0\\0& I_m\end{bmatrix}\begin{bmatrix}
y\\ u
\end{bmatrix}=\|u\|^2-\|y\|^2.
\end{align}
The following notion of discrete-time dissipative descriptor systems was used for standard state-space systems in \cite[Appendix C]{GooS84}, see also \cite{Sta03}, and can be viewed as a discrete-time analog of the classic definition for continuous-time systems \cite{Wil72}.
\begin{definition}\label{def:disssyst}
A discrete-time descriptor system of the form \eqref{discr_DAE} with coefficients $(E,A,B,C,D)$  is called \emph{dissipative} with respect to the supply rate $s$ given by \eqref{def:supply}, if there exists a nonnegative function $V:\K^n\rightarrow[0,\infty)$ satisfying $V(0)=0$ and the \emph{dissipation inequality}
\begin{align}
\label{eq:dissip}
V(Ex_{k+1})-V(Ex_k)\leq s(u_k,y_k) \quad \text{for all $k\geq 0$}
\end{align}
and all consistent combinations of $u_k\in\K^m$ and $x_0\in\K^n$. Furthermore, such functions $V$ are called \emph{storage functions}. The system is called \emph{strictly dissipative} if \eqref{eq:dissip} holds with strict inequality for all consistent $x_0\neq 0$ and $u_k$ and the resulting $y_k$, $k\geq 0$. Moreover, the system is called \emph{conservative} if \eqref{eq:dissip} holds with equality.
\end{definition}
Note that for standard state-space systems, i.e. for $E=I$, we obtain the classic definition from \cite{GooS84}. 

If the system is dissipative with respect to $s_{imp}$, then the system is called \emph{impedance passive} and  if the system is dissipative with respect to $s_{sca}$, then the system is called \emph{scattering passive}.

For quadratic supply rates, it is common to consider also quadratic storage functions
\begin{align}
    \label{def:quad_store}
V:\K^n\rightarrow[0,\infty),\quad  V(x):=x^HXx\quad \text{for some $X\in\K^{n\times n}$ with  $X=X^H\geq 0$.}
\end{align}
For this case, we introduce the following two conditions for  impedance and scattering passivity.
\begin{itemize}

\item[] \hspace{-0.5cm} {\rm (d-iPa)}~
There exists a storage function $V$ of the form \eqref{def:quad_store} satisfying 
\begin{equation}
V(Ex_{k+1})-V(Ex_k)\leq s_{imp}(u_k,y_k)=2\Re (u_k^Hy_k)\quad \text{for all $k\geq 0$.}   \label{eq:diPA}
\end{equation}
\item[]\hspace{-0.5cm}  {\rm (d-sPa)}~
There exists a storage function $V$ of the form \eqref{def:quad_store} satisfying
\begin{equation}
 V(Ex_{k+1})-V(Ex_k)\leq s_{sca}(u_k,y_k)=\|u_k\|^2-\|y_k\|^2\quad \text{for all $k\geq 0$.}   \label{eq:dsPA}
\end{equation}
\end{itemize}

After extending the classical passivity notions to discrete-time descriptor systems, in the next section, we will characterize these properties via linear KYP type matrix inequalities (LMIs).

\subsection{Characterization of dissipativity via LMIs}

The discussed passivity notions lead to the following variants of the KYP inequalities. Different generalizations of KYP inequalities for discrete-time descriptor systems were introduced in \cite{BanV19}.

For the impedance passive supply rate, we consider the LMI for matrices $ X=X^H\in\K^{n\times n}$
\begin{align}
\label{eq:diKYP}
\begin{split}
&   \hspace{-1.5cm} \text{(d-iKYP)}~\ \mbox {there exists}~X=X^H\geq 0~\mbox{s.t.}\\ & ~~~~\begin{bmatrix}
-A^HXA+E^HXE&C^H-A^HXB\\C-B^HXA&D+D^H-B^HXB
\end{bmatrix}\geq 0. \hspace{1.3cm}
\end{split}
\end{align}
For standard state-space systems, i.e. $E=I$, this inequality coincides with the KYP inequality considered in \cite{AndH69,XiaH99}.  Note also that for $D=0$ in this LMI we need to have that $B^HXB=0$.

For the scattering passive supply rate, we consider the LMI for $X=X^H\in\K^{n\times n} $
\begin{align}
\label{eq:dsKYP}
\begin{split}
&   \hspace{-0.7cm} \text{(d-sKYP)}~\ \mbox {there exists}~ X=X^H\geq 0~\mbox{s.t.}\\ & \hspace{1.3cm}\begin{bmatrix}
-A^HXA+E^HXE-C^HC&-A^HXB-C^HD\\-D^HC-B^HXA& I-D^HD- B^HXB
\end{bmatrix}\geq 0. \hspace{0cm}
\end{split}
\end{align}

For standard state-space systems, i.e. $E=I$, this inequality coincides with the KYP inequality considered in \cite{Vai1985,XiaH99}, see also \cite[Section 5.10.2]{BroLME20}. 

In the following proposition it is shown that the two KYP LMIs characterize passivity.
\begin{proposition}
\label{prop:KYPthenPa}
If \emph{(d-iKYP)} (respectively \emph{(d-sKYP)}) holds for some $X=X^H\geq 0$ then \emph{(d-iPa)} (respectively \emph{(d-sPa)}) holds for some $X=X^H\geq 0$.
\end{proposition}

\begin{proof}
    The proof of Proposition~\ref{prop:KYPthenPa} is presented in Section~\ref{app:KYPthenPa}.
\end{proof}

The following example illustrates the difficulties that arise if Assumption~\ref{Assumpt} does not hold.
\begin{example}
Consider $E=\begin{bmatrix}0&1\\0&0\end{bmatrix}$ and $A=\begin{bmatrix}
1&0\\0&1
\end{bmatrix}$ so that the associated discrete-time descriptor system~\eqref{dae_discr} has index two. Then (d-iKYP) has no solution $0<X=X^H=\begin{bmatrix}
x_{11}&x_{12}\\ x_{21}& x_{22}
\end{bmatrix}$. If this were the case, then 
\[
0\leq -A^HXA+E^HXE=\begin{bmatrix}
-x_{11}&-x_{12}\\-x_{12} &x_{11}-x_{22}
\end{bmatrix}.
\]
Hence $x_{11}=0$ which implies $x_{12}=0$ and thus $x_{22}=0$. Consequently, $X=0$ and therefore also $C=0$ holds. 

On the other hand, if we define $C=\begin{bmatrix}
0&1
\end{bmatrix}$, $D=0$, then using \eqref{dae_discr}, we find $y_k=Cx_k=0$ and therefore the system is passive with the trivial storage function $V=0$. However, as we have seen, the system does not fulfill (d-iKYP). 
\end{example}

From passivity with respect to a certain supply rate one can only conclude that the corresponding KYPs hold on the system space, see \cite{CheGH22} for the continuous-time case. The solvability of such KYP inequalities restricted to subspaces was studied in \cite{CamF09,ReiRV15,ReiS10} for continuous-time systems and in \cite{BanV19} for discrete-time systems.

\begin{corollary}\label{cor: pakyp}
Consider a discrete-time descriptor system in semi-explicit, index-one form \eqref{eq:semi_expl} and consider 
\begin{align*}
\Ac&:=\Sigma_E^{-\tfrac{1}{2}}(A_{11}-A_{12}A_{22}^{-1}A_{21})\Sigma_E^{-\tfrac{1}{2}},&\Bc&:=\Sigma_E^{-\tfrac{1}{2}}(B_1-A_{12}A_{22}^{-1}B_2),\\ \Cc&:=(C_1-C_2A_{22}^{-1}A_{21})\Sigma_E^{-\tfrac{1}{2}},& \Dc&:=D-C_2A_{22}^{-1}B_2.
\end{align*}
Then \emph{(d-iPa)} holds if and only if there exists a solution $0\leq X_1=X_1^H$ for \emph{(d-iKYP)} associated with  the standard state-space system $(I_n,\Ac,\Bc,\Cc,\Dc)$.
%the following KYP inequality has a solution
%\begin{align}
  %  \label{eq:KYP_indexone}
%\begin{bmatrix}
%- \Ac^H\mathcal{Q}\Ac+\Sigma_E\mathcal{Q}\Sigma_E& \Cc^H - \Ac^H\mathcal{Q}\Bc \\
%\Cc-\Bc^H\mathcal{Q}\Ac
%&\Dc+\Dc^H-\Bc \mathcal{Q}\Bc
%\end{bmatrix}\geq 0,\quad \mathcal{Q}=\mathcal{Q}^H\geq 0.
%\end{align}

Furthermore, \emph{(d-sPa)} holds if and only if \emph{(d-sKYP)}, for the standard state space system $(I_n,\Ac,\Bc,\Cc,\Dc)$, has a solution $\mathcal{X}\geq 0$.
%the following LMI has a solution
%\begin{align}
  %  \label{eq:KYP_indexone_BR}
%\begin{bmatrix}
%- \Ac^H\mathcal{Q}\Ac+\mathcal{Q}-\Cc^H\Cc& -\Cc^H\Dc - \Ac^H\mathcal{Q}\Bc \\
%\Cc-\Bc^H\mathcal{Q}\Ac
%&I_m-\Dc^H\Dc-\Bc \mathcal{Q}\Bc
%\end{bmatrix}\geq 0,\quad \mathcal{Q}=\mathcal{Q}^H\geq 0
%\end{align}
\end{corollary}
\begin{proof}
    The proof of Corollary~\ref{cor: pakyp} is presented in Section~\ref{app:pakyp}. 
\end{proof}

%\HGcomment{Add passive index 2 examples such KYP is not solvable}
%We study the dissipativity of these systems with respect to given supply rates \textbf{see Volkers definition}.
%We could study the characterization via available storage in this general setting and then later go to the special case of passive systems.
%\HGcomment{I think we have to distinguish between space and time (or space time?) discretization methods}
%, however we will also consider the DAEs where $E$ is singular to allow for a wider class of discretization schemes.

In this subsection, we have presented characterizations for scattering and impedance passive systems using KYP type LMIs.

In the next section, we show the relationship to bounded and positive real systems.

\subsection{Characterization of dissipativity via transfer functions}\label{sec:transfer}
In this subsection, we consider discrete-time systems of the form \eqref{discr_DAE} with coefficients $(E,A,B,C,D)$ and their transfer function \eqref{tf}.
%which can be derived from the $z$-transform and is %given by 
%
%\[
%\Tc(z)=C(zE-A)^{-1}B+D
%\]
%
%for all $z$ for which the inverse $(zE-A)^{-1}$ exists. %For $z\neq 0$ we can use the equivalent representation $T(z)=Cz(E-zA)^{-1}B+D$.
%The results on minimality and realization of index at most two can be formulated as well in the discrete-time case.
Consider the following classical definition, see \cite{AndH69,XiaH99}.
\begin{definition}\label{def:bdposreal}
A rational function $\Tc:\Omega\rightarrow\dC^{m\times m}$ on some domain $\Omega\subset\dC$ is called \emph{positive real} if
\begin{itemize}
\item[] \hspace{-0.5cm} {\rm (d-PR)}~ $\Tc(\cdot)$ has analytic entries and $\Tc(z)+\Tc(z)^H\geq 0$ for all $z\in\dC$ with $\vert z\vert>1$,
\end{itemize}
and it is called \emph{bounded real}, see \cite{And67},  if 
\begin{itemize}
\item[] \hspace{-0.5cm} {\rm (d-BR)}~ $\Tc(\cdot)$ has analytic entries and $I_m-\Tc(z)^H\Tc(z)\geq 0$ for all $z\in\dC$ with $\vert z\vert>1$.
\end{itemize}
\end{definition}
In \cite{AndH69} it is shown that for minimal standard discrete-time state-space systems (d-PR) and (d-iKYP) are equivalent. Another related definition of bounded realness was studied in  \cite[Section 5.10.2]{BroLME20}. 

If $\Tc$ is the transfer function of an impedance passive system then one has the following relation between inputs and outputs, see also Section~\ref{sec: externalcayley}, 
\[
\tfrac{1}{\sqrt{2}}(e-f)=y=\Tc(z)u=\Tc(z)\tfrac{1}{\sqrt{2}}(e+f),\quad f=(I_n+\Tc(z))^{-1}(I_n-\Tc(z))e.
\]
Since $\Tc(z)+\Tc(z)^H\geq 0$ for all $z\in\dC$ with $\vert z\vert>1$ we have that $I_n+\Tc(z)$ is invertible. Furthermore, one can show that $(I_n+\Tc(z))^{-1}(I_n-\Tc(z))$ is bounded real \cite[Theorem 2.3]{ReiS10}  using the additional assumption that $I_m+D$ is invertible. Conversely, it was also shown in \cite[Theorem 2.3]{ReiS10} that for a bounded real transfer function $\Tc$ such that $I_m+\Tc(z)$ is invertible for all $\vert z\vert>1$ then the transformed transfer function is positive real. 

The following lemma is an extension of \cite[Lemma 3.2]{BanV19}.
\begin{lemma}
\label{lem:KYP_without_supply}
Consider a regular discrete-time descriptor system of the form \eqref{discr_DAE}.
%satisfying Assumption~\ref{Assumpt}. 
Let $X=X^H\in\dC^{n\times n}$ then for all $z\in\dC\setminus\sigma(E,A)$ the following identity holds 
 \begin{align*}
 &~~~~\begin{bmatrix}
    (zE-A)^{-1}B\\ I_m
    \end{bmatrix}^H\begin{bmatrix}
    -A^HXA+E^HXE& -A^HXB\\ -B^HXA &-B^HQB
    \end{bmatrix}\begin{bmatrix}
    (zE-A)^{-1}B\\ I_m
    \end{bmatrix}\\&=(1-\vert z\vert^2)B^H(zE-A)^{-H}E^HXE(zE-A)^{-1}B.
 \end{align*}
\end{lemma}
\begin{proof}
    The proof of Lemma~\ref{lem:KYP_without_supply} is presented in Section~\ref{app:KYP_without_supply}. 
\end{proof}

We then have the following relation between impedance (scattering) passivity and positive (bounded) realness.
 \begin{proposition}\label{prop: brpa} Consider a discrete-time descriptor system of the form \eqref{discr_DAE} with coefficients $(E,A,B,C,D)$ satisfying Assumption~\ref{Assumpt}. If {\rm (d-iPa)} (resp.\ {\rm (d-sPa)}) holds, then the system is positive real (resp.\ bounded real).  
 \end{proposition}
\begin{proof}
    The proof of Proposition~\ref{prop: brpa} is presented in Section~\ref{app:brpa}. 
\end{proof}

To see where controllability and observability conditions come into play when considering impedance (scattering) passive systems, we need the following lemma.

\begin{lemma}
\label{lem:TF}
Consider a descriptor system \eqref{discr_DAE} 
with coefficients $(E,A,B,C,D)$ satisfying Assumption~\ref{Assumpt} and the associated standard-state space system \eqref{eq:indexone_ODE} with coefficients $(I_n,\Ac,\Bc,\Cc,\Dc)$. Then the following statements hold:
\begin{itemize}
\setlength{\itemindent}{1em}
    \item[\rm (a)] The transfer function of \eqref{discr_DAE} coincides with the transfer function of \eqref{eq:indexone_ODE}.
    \item[\rm (b)] The discrete-time descriptor system with coefficients $(E,A,B,C,D)$ is behaviorally controllable (observable) if and only if the standard state space system with coefficients $(I_n,\Ac,\Bc,\Cc,\Dc)$ is controllable (observable).
\end{itemize}
\end{lemma}
\begin{proof}
    The proof of Lemma~\ref{lem:TF} is presented in Section~\ref{app:TF}. 
\end{proof}
As an immediate consequence of  Lemma~\ref{lem:TF} and \cite[Corollary 13.14]{HadC11} we obtain that for behaviorally controllable discrete-time descriptor systems, i.e.\ systems that satisfy (C1), positive (resp.\ bounded) realness implies impedance (resp.\ scattering) passivity.
%
%\HGcomment{turn into corollary}
\begin{corollary}
\label{cor:brkyp}
Consider a~behaviorally controllable discrete-time descriptor systems of the form \eqref{discr_DAE} satisfying Assumption~\ref{Assumpt}. If the system satisfies \emph{(d-PR)} (resp.\ \emph{(d-BR)}), then the associated standard state-space system \eqref{eq:indexone_ODE} satisfies \emph{(d-iKYP)} (resp.\ \emph{(d-sKYP)}).
\end{corollary}
In Corollary~\ref{cor:brkyp} we cannot drop the controllability assumption (C1) as the following examples show.
\begin{example}\label{ex:countbrkyp}
The discrete-time (standard state-space) system \eqref{discr_DAE} with coefficients $(E,A,B,C,D)=(1,\tfrac12,0,1,0)$ is stable and not behaviorally controllable. The system fulfills $\mathcal{T}(s)=C(zE-A)^{-1}B+D=0$ and is therefore positive real, but (d-iKYP) has no solution. 

Consider the discrete-time (standard state-space) system \eqref{discr_DAE} with coefficients  $(E,A,B,C,D)=(1,\tfrac12,0,1,1)$ is stable and not behaviorally controllable. The system fulfills $\mathcal{T}(s)=C(zE-A)^{-1}B+D=1$ and is therefore bounded real, but (d-sKYP) has no solution. 

These two examples also show that the controllability assumption (C1) cannot be relaxed to a stabilizability assumption. 
\end{example}

For standard discrete-time state-space systems, the equivalence between positive (bounded) realness properties and solvability of the corresponding KYPs can be proved under additional strictness assumptions of the involved inequalities, see \cite{LeeC00}. In particular, there it was shown that (d-iKYP) (or (d-sKYP)) with strict inequalities is equivalent to strict positive (bounded) realness, respectively,  i.e.\ (d-PR) and (d-BR) holds with positive definiteness and for all $z\in\dC$ with $\vert z\vert \geq 1$.

For related results on the relation between bounded realness and positive realness and solutions of KYP inequalities on certain subspaces for descriptor systems with index higher than one we refer to \cite{BanV19,ReiS10}.

In this section, we have presented the relationship between impedance (scattering) passivity, positive (bounded) realness for dissipative systems, and the solvability of KYP LMIs. In the next section, we derive the relation of these properties to discrete-time port-Hamiltonian systems.

\section{Port-Hamiltonian representation of discrete-time dissipative systems}
\label{sec:phDef}
In this section, we propose an extension of the port-Hamiltonian (pH) framework to discrete-time descriptor systems using the scattering supply rate rather than the impedance supply rate which is typically used  to define continuous-time pH systems. 

Recall that we consider only descriptor systems of the form \eqref{discr_DAE} with coefficients $(E,A,B,C,D)$ that are completely causal, i.e satisfy Assumption~\ref{Assumpt}, which implies by Proposition \ref{prop:causal} that the system has an index at most one. Then 
 we can alternatively consider the associated standard state-space system \eqref{eq:indexone_ODE} with coefficients $(I_n,\Ac,\Bc,\Cc,\Dc)$. Note that if $E=I_n$ then we trivially have $(I_n,\Ac,\Bc,\Cc,\Dc)=(I_n,A,B,C,D)$.

For the definition of standard discrete-time pH systems, we consider for some $X=X^H>0\in\K^{n\times n}$ the \emph{weighted Euclidean norm}
\[
\left\|\begin{bmatrix}
 x\\ u
 \end{bmatrix}\right\|_X:= \left\|\begin{bmatrix} X^{\frac 12}&0\\ 0&I_m\end{bmatrix}\begin{bmatrix}
 x\\ u
 \end{bmatrix}\right\|,
\]
where $\|\cdot\|$ denotes the standard Euclidean norm on $\K^n\times\K^m$. We also use the \emph{weighted spectral norm}
 \[
\left\|\begin{bmatrix}
 A&B\\ C&D
 \end{bmatrix}\right\|_X:=\sup_{(x,u)\neq 0}\frac{\left\|\begin{smallbmatrix}
 A&B\\ C&D
 \end{smallbmatrix}\begin{smallbmatrix}
 x\\ u
 \end{smallbmatrix}\right\|_X}{\|\begin{smallbmatrix}
 x\\ u
 \end{smallbmatrix}\|_X}.
 \]
We then propose the following definition of discrete-time standard state-space and descriptor pH systems.
 \begin{definition}\label{def:discpH} 
 A standard discrete-time state-space system of the form \eqref{discr_DAE} with coefficients $(I_n,A,B,C,D)$ is called standard \emph{discrete-time scattering pH system} (d-spH) if there exists $X=X^H>0$ such that %the system  $(I_n,\tilde{A},\tilde{B},\tilde{C},D)$ given by \eqref{eq:scat_sys_trafo} and all singular values of $\begin{bmatrix}
% \tilde{A} & \tilde{B}\\ \tilde{C} & D
% \end{bmatrix}$ are in $[0,1]$ i.e. 
\[
\Big\| \begin{bmatrix}
 A & B\\ C & D
 \end{bmatrix}\Big\|_X \leq 1.
 \]

A completely causal discrete-time descriptor system of the form \eqref{discr_DAE} with coefficients $(E,A,B,C,D)$  and the associated standard discrete-time state-space system \eqref{eq:indexone_ODE} with coefficients $(I_n,\Ac,\Bc,\Cc,\Dc)$ is called \emph{discrete-time scattering pH descriptor system} if there exists $\mathcal X=\mathcal X^H>0$ such that 
\[
\Big\| \begin{bmatrix}
 \Ac & \Bc\\ \Cc & \Dc
 \end{bmatrix}\Big\|_{\mathcal X} \leq 1.
 \]
 \end{definition}
 
 Here $X$ ($\mathcal X$) is the weight matrix which can be used to define the Hamiltonian of the system according to $\Hc(x):=\tfrac12x^HXx$ ($\Hc(x):=\tfrac12x^H\mathcal X x$). This Hamiltonian is also a Lyapunov function because Definition~\ref{def:discpH} implies that the following Lyapunov inequality holds 
 \[
 -A^H X A + X \geq 0 \quad (\text{resp.\ }\,  -\mathcal{A}^H \mathcal X \mathcal{A} + \mathcal X \geq 0).
 \] 
Hence, Proposition \ref{prop:stable} which is presented in Section~\ref{sec:stable} yields that discrete-time pH (descriptor) systems are stable. 
 
 The  pH formulation for continuous-time descriptor systems also allows for semidefinite Hamiltonians $\mathcal{H}$. However, this may lead to pH descriptor systems that are unstable without imposing further assumptions, see \cite{MehMW18,GerH21}. The same is true for discrete-time systems, which is why we restrict ourselves to positive definite matrices $X$ ($\mathcal X$) in Definition~\ref{def:discpH}.
 
The following proposition shows that discrete-time pH descriptor systems defined in this way are  scattering passive. Furthermore, for a discrete-time pH descriptor system with a solution $0<X=X\in\K^{n\times n}$ we may also consider the following transformed standard state-space system
\begin{align}
\label{eq:scat_sys_trafo}
\begin{split}
     \underbrace{X^{\frac{1}{2}} x_{k+1}}_{\widetilde{x}_{k+1}} &= \underbrace{X^{\frac{1}{2}} A X^{-\frac{1}{2}}}_{=:\widetilde{A}} \underbrace{X^{\frac{1}{2}} x_k}_{\widetilde{x}_k} + \underbrace{X^{\frac{1}{2}} B}_{=:\widetilde{B}} u_k\\
     y_k &= \underbrace{C X^{-\frac{1}{2}}}_{=:\widetilde{C}} \underbrace{X^{\frac{1}{2}} x_k}_{\widetilde{x}_k} + D u_k.
     \end{split}
 \end{align}
with the same transfer function, i.e.\ $\widetilde C(zI_n-\widetilde A)^{-1}\widetilde B+D=C(zI_n-A)^{-1}B+D$.

 \begin{proposition}
 \label{prop:pHandKYP}
Consider a standard discrete-time  state-space system with coefficients  $(I_n,A,B,C,D)$. If \emph{(d-sKYP)} has a solution $0<X=X^H$ then the transformed system given by \eqref{eq:scat_sys_trafo} satisfies
\begin{align}
\label{eq:norm_estimate}
 \left\|\begin{bmatrix}
 A&B\\ C&D
 \end{bmatrix}\right\|_X= \left\| \begin{bmatrix}\widetilde A& \widetilde B\\ \widetilde C&  D \end{bmatrix}\right\|\leq 1.
 \end{align}
 Conversely, if \eqref{eq:norm_estimate} holds for some $0<X=X^H$ then this $X$ is a solution to \emph{(d-sKYP)}.
 \end{proposition}
 \begin{proof}
Using a congruence transformation, we see that the inequality (d-sKYP) is equivalent to
 \begin{align*}
     &\begin{bmatrix}
     - X^{-\frac{1}{2}} A^H X^{\frac{1}{2}}X^{\frac{1}{2}}AX^{-\frac{1}{2}}+I-X^{-\frac{1}{2}} C^H C X^{-\frac{1}{2}}
     & - X^{-\frac{1}{2}} A^H X^{\frac{1}{2}} X^{\frac{1}{2}}B- X^{-\frac{1}{2}} C^H D\\
     -D^H C X^{-\frac{1}{2}} - B^H X^{\frac{1}{2}}X^{\frac{1}{2}} A X^{-\frac{1}{2}} & I -D^H D - B^H X^{\frac{1}{2}} X^{\frac{1}{2}}B
     \end{bmatrix} \\
     &= \begin{bmatrix}
     - \widetilde{A}^H \widetilde{A} + I - \widetilde{C}^H \widetilde{C} & - \widetilde{A}^H \widetilde{B}- \widetilde{C}^H D\\
     - \widetilde{B}^H \widetilde{A} - D^H \widetilde{C} & I - D^H D - \widetilde{B}^H \widetilde{B}
     \end{bmatrix}\\
     &= I - \begin{bmatrix}
     \widetilde{A} &\widetilde{B}\\
     \widetilde{C} &D
     \end{bmatrix}^H \begin{bmatrix}
     \widetilde{A}& \widetilde{B}\\
     \widetilde{C}& D
     \end{bmatrix}
     \geq 0.
 \end{align*}
 This inequality can be written as
 \begin{equation*}
 \left\|\begin{bmatrix}
     \widetilde{A} &\widetilde{B}\\
     \widetilde{C} &D
     \end{bmatrix} x \right \|^2 =
     x^H \begin{bmatrix}
     \widetilde{A} &\widetilde{B}\\
     \widetilde{C} &D
     \end{bmatrix}^H \begin{bmatrix}
     \widetilde{A}& \widetilde{B}\\
     \widetilde{C}& D
     \end{bmatrix} x \leq x^H x = \| x \|^2 \quad \text{for all } x \in \K^{n+m},
 \end{equation*}
 and therefore we get
 \begin{equation*}
    \left\|\begin{bmatrix}
     \widetilde{A} &\widetilde{B}\\
     \widetilde{C} &D
     \end{bmatrix} \right \| = \max_{\substack{x \in \K^{n+m}\\ x \neq 0}} \frac{\left\|\begin{bmatrix}
     \widetilde{A} &\widetilde{B}\\
     \widetilde{C} &D
     \end{bmatrix} x \right \|}{\| x \|} \leq 1.
 \end{equation*}
 It remains to verify the equality in \eqref{eq:norm_estimate}. Consider 
 \begin{align*}
 \left\|\begin{bmatrix}
 A&B\\ C&D
 \end{bmatrix}\right\|_X&=\sup_{(x,u)\neq 0, \|(x,u)\|_X=1}\left\|\begin{bmatrix}
 A&B\\ C&D
 \end{bmatrix}\begin{bmatrix}
 x\\ u
 \end{bmatrix}\right\|_X\\&=\sup_{(x,u)\neq 0, \|(X^{\tfrac{1}{2}}x,u)\|=1}\left\|\begin{bmatrix}
 X^{\tfrac{1}{2}}A&X^{\tfrac{1}{2}} B\\ C&D
 \end{bmatrix}\begin{bmatrix}
 x\\ u
 \end{bmatrix}\right\|\\
 &=\sup_{(x,u)\neq 0, \|(X^{\tfrac{1}{2}}x,u)\|=1}\left\|\begin{bmatrix}
 X^{\tfrac{1}{2}}AX^{-\tfrac{1}{2}}&X^{\tfrac{1}{2}} B\\ 
 CX^{-\tfrac{1}{2}}&D
 \end{bmatrix}\begin{bmatrix}
 X^{\tfrac{1}{2}}x\\ u
 \end{bmatrix}\right\|
\\
 &=\left\|\begin{bmatrix}
     \widetilde{A} &\widetilde{B}\\
     \widetilde{C} &D
     \end{bmatrix} \right \|.
     \end{align*}
\end{proof}
Although Proposition~\ref{prop:pHandKYP} is only stated for standard state-space systems, from Definition~\ref{def:discpH} and Corollary~\ref{cor: pakyp} it follows that discrete-time pH descriptor systems fulfill (d-sKYP) and that they are scattering passive. This is summarized in the following corollary.
%.
\begin{corollary}\label{cor: kyptoph}
Consider a completely causal descriptor system \eqref{discr_DAE} (i.e. satisfying Assumption~\ref{Assumpt}) with coefficients $(E,A,B,C,D)$. Then the system is scattering pH in the sense of Definition~\ref{def:discpH} if and only if \emph{(d-sKYP)} for the standard state-space system with coefficients $(I_n,\Ac,\Bc,\Cc,\Dc)$ has a solution $\mathcal{X}=\mathcal{X}^H>0$. Furthermore, the descriptor system with coefficients $(E,A,B,C,D)$ is scattering passive. 
\end{corollary}

Given a bounded real rational function, then by Corollary~\ref{cor:brkyp} this function has a minimal realization as a standard state-space system which is scattering passive. 
 Hence, using Proposition~\ref{prop:pHandKYP} we obtain that every bounded real rational function can be realized as a standard discrete-time state-space  pH system.

The following result from \cite[Theorem 13.18]{HadC11} shows that every behaviorally observable scattering passive discrete-time system, i.e. satisfying (O1), is pH. Furthermore, by Lemma~\ref{lem:TF} this result trivially extends to behaviorally observable completely causal discrete-time descriptor systems.
\begin{proposition}\label{prop: observqposdif }
%Consider the system given by $(I_n,A,B,C,D)$. If $(A,C)$ is observable and let $Q\geq 0$ be a solution of (d-sKYP) or (d-iKYP) then $Q$ is positive definite. 
Consider a completely causal and behaviorally observable discrete-time descriptor system of the form \eqref{discr_DAE}  with coefficients $(E,A,B,C,D)$ and  the associated standard discrete-time state-space system \eqref{eq:indexone_ODE} with coefficients $(I_n,\Ac,\Bc,\Cc,\Dc)$. Then every solution $X=X^H\in\K^{n\times n}$ of \emph{(d-iKYP)} (or \emph{(d-sKYP)}) for $(I_n,\Ac,\Bc,\Cc,\Dc)$ is positive definite. 
\end{proposition}
The following example shows the condition of behavior observability (condition (O1)) cannot be omitted, i.e. that not every scattering passive discrete-time descriptor system can be expressed equivalently as discrete-time pH system, see  \cite{CheGH22} for a similar example for continuous-time impedance passive systems.
\begin{example}\label{ex: counterkypph}
Consider the discrete-time system with coefficients $(E,A,B,C,D)= (1, \tfrac{1}{2}, \tfrac{1}{2}, 0,1)$. This system is not observable and the matrix
\begin{align*}
    \begin{bmatrix}
    -A^H X A +E^H X E - C^H C & -A^H X B - C^H D\\
    -D^H C -B^H X A & 1- D^H D - B^H X B
    \end{bmatrix} =
    \tfrac{1}{4} \begin{bmatrix}
    3 X & - X \\ -X & -X
    \end{bmatrix}
\end{align*}
is positive semidefinite only if $X=0$ and then it satisfies the (d-sKYP).
%\KCcomment{This is not consistent. Since we said that in (d-sKYP) that X is definite and not semidefinite. Also in the figure we assume that we are the semidefinite case since we need observability. So I would suggest making it semidefinite in the definition of KYP}
However, for $X=X^H>0$
 \begin{align*}
     \left\|\begin{bmatrix}
 A&B\\ C&D
 \end{bmatrix}\right\|_X
 &=  \left\|\begin{bmatrix}
 \tfrac{1}{2}& \tfrac{1}{2}\\ 0&1
 \end{bmatrix}\right\|_X
 = \sup_{(x,u)\neq 0} \frac{\left\|\begin{bmatrix}
 \tfrac{1}{2} (x +u)\\ u
 \end{bmatrix}\right\|_X}{ \left\|\begin{bmatrix}
 x\\ u
 \end{bmatrix}\right\|_X}
 = \sup_{(x,u)\neq 0} \frac{\left\|\begin{bmatrix}
 \tfrac{1}{2} X^{\frac 12} (x +u)\\ u
 \end{bmatrix}\right\|}{ \left\|\begin{bmatrix}
X^{\frac 12} x\\ u
 \end{bmatrix}\right\|}\\
 &= \sup_{(x,u)\neq 0} \frac{\tfrac{1}{4}X (x+u)^2 + u^2}{Xx^2+u^2}
 = \sup_{(x,u)\neq 0} 
 \frac{\tfrac{1}{4}X (x+u)^2 -Xx^2}{Xx^2+ u^2} +1 > 1.
 \end{align*}
Therefore, it is not a discrete-time scattering pH system.
\end{example}

We summarize the subtle relationship between the different passivity related concepts for discrete-time passive systems in Figure~\ref{fig:overvieweinvertible} and Figure~\ref{fig:overvieweinvertible2}. 

   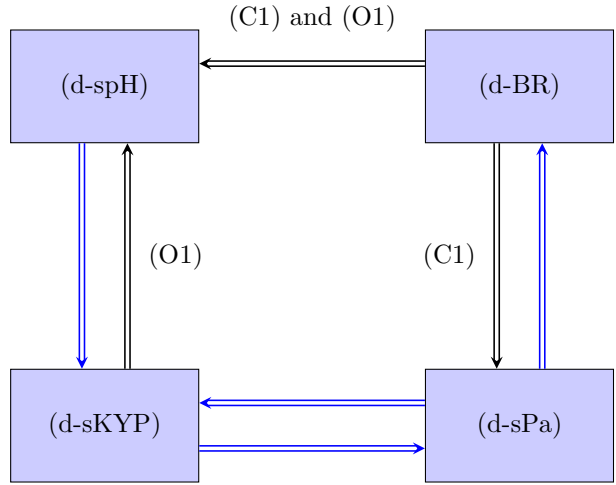
\begin{figure}[htbp!]
    \centering
        \begin{tikzpicture}
        \node[block] (a) {(d-spH)};
        \node[block, below =3cm of a]   (b){(d-sKYP)};
        \node[block, right =3cm of b]   (c){(d-sPa)};
        \node[block, above =3cm of c]   (d){(d-BR)};
        
        \draw[->,semithick,double,double equal sign distance,>=stealth, color=blue] ([xshift=-2ex]a.south) -- ([xshift=-2ex]b.north);
        \draw[->,semithick,double,double equal sign distance,>=stealth, color=black] ([xshift=2ex]b.north) -- ([xshift=2ex]a.south) node[midway,right = 1 ex]{(O1)};
        %\draw[->,semithick,double,double equal sign distance,>=stealth] (b.west) --  ++(-10pt,0) coordinate[yshift=-1.7 cm,](r){} |- (a.west)node[midway,below left= 0.75cm and 1ex]{\parbox{2cm}{\begin{center} observable\end{center}}};
        \draw[->,semithick,double,double equal sign distance,>=stealth, color=blue] ([yshift=-2ex]b.east) -- ([yshift=-2ex]c.west) node[midway,below = 1 ex]{};
        \draw[->,semithick,double,double equal sign distance,>=stealth, color=blue] ([yshift=2ex]c.west) -- ([yshift=2ex]b.east) node[midway,above = 1 ex]{ };
        \draw[->,semithick,double,double equal sign distance,>=stealth, color=blue]([xshift=2ex]c.north) -- ([xshift=2ex]d.south);
        \draw[->,semithick,double,double equal sign distance,>=stealth,color=black] ([xshift=-2ex]d.south) -- ([xshift= -2ex]c.north) node[midway,left = 1 ex]{ (C1) };
        %\draw[->,semithick,double,double equal sign distance,>=stealth, color=black] (d.east) --  ++(10pt,0) coordinate[yshift=-1.7cm,](r){} |- (c.east)node[midway,below right= -2 cm and 1ex]{controllable};
        \draw[->,semithick,double,double equal sign distance,>=stealth] ([yshift=2ex]d.west) -- ([yshift=2ex]a.east)node[midway,above = 2 ex]{(C1) and (O1)};
          \end{tikzpicture}
    \caption{The relationship between (d-spH), (d-sKYP), (d-sPa) and (d-BR) for discrete-time descriptor systems satisfying Assumption~\ref{Assumpt}. The color blue marks implication without additional assumptions and the color black implications with additional assumptions.\\
       (C1) behaviorally controllable,
       (O1) behaviorally observable,}
    \label{fig:overvieweinvertible}
\end{figure}  
   \begin{figure}[htbp!]
    \centering
    \scalebox{0.8}{
        \begin{tikzpicture}
        \node[block] (a) {(d-spH)};
        \node[block, below =3cm of a]   (b){(d-sKYP)};
        \node[block, right =4cm of b]   (c){(d-sPa)};
        \node[block, above =3cm of c]   (d){(d-BR)};
        
        \draw[->,semithick,double,double equal sign distance,>=stealth, color=blue] ([xshift=-2ex]a.south) -- ([xshift=-2ex]b.north)node[midway, below left=-2ex and  1 ex]{ Cor. \ref{cor: kyptoph}};
        \draw[->,semithick,double,double equal sign distance,>=stealth] ([xshift=2ex]b.north) -- ([xshift=2ex]a.south) node[midway, below right=-4ex and  1 ex]{\parbox{1.5cm}{\begin{center} Cor. \ref{cor: kyptoph} \\and\\Prop. \ref{prop: observqposdif }  \end{center}}};
        \draw[->,semithick,double,double equal sign distance,>=stealth,negated, color=red] (b.west) --  ++(-30pt,0) coordinate[yshift=-1.7 cm,](r){} |- (a.west)node[midway,below left= 10.5ex and 1ex]{\parbox{1.5cm}{\begin{center} Ex.~\ref{ex: counterkypph}\end{center}}};
        \draw[->,semithick,double,double equal sign distance,>=stealth, color=blue] ([yshift=-2ex]b.east) -- ([yshift=-2ex]c.west)node[midway,below = 1 ex]{Prop.~\ref{prop:KYPthenPa} or Cor.~\ref{cor: pakyp}};
        \draw[->,semithick,double,double equal sign distance,>=stealth, color=blue] ([yshift=2ex]c.west) -- ([yshift=2ex]b.east) node[midway,above = 1 ex]{Cor.~\ref{cor: pakyp} };
        \draw[->,semithick,double,double equal sign distance,>=stealth, color=blue]([xshift=2ex]c.north) -- ([xshift=2ex]d.south)node[midway,right = 1 ex]{Prop.~\ref{prop: brpa}~~};
        \draw[->,semithick,double,double equal sign distance,>=stealth] ([xshift=-2ex]d.south) -- ([xshift=-2ex]c.north) node[midway,left = 1 ex]{Cor.~\ref{cor:brkyp}~~};
        \draw[->,semithick,double,double equal sign distance,>=stealth, negated,color=red] (d.east) --  ++(30pt,0) coordinate[yshift=-1.7cm,](r){} |- (c.east)node[midway,below right= -2.5 cm and 1ex]{Ex.~\ref{ex:countbrkyp}};
        %\draw[->,semithick,double,double equal sign distance,>=stealth] ([yshift=2ex]d.west) -- ([yshift=2ex]a.east)node[midway,above = 2 ex]{minimal, missing};
          \end{tikzpicture}
          }
    \caption{The relationship between (d-spH), (d-sKYP), (d-sPa) and (d-BR) for  discrete-time descriptor systems satisfying Assumption~\ref{Assumpt}. The color blue marks implication without additional assumptions and the color black implications with additional assumptions.  Counterexamples for the case that assumptions are not fulfilled are colored in red. }
    \label{fig:overvieweinvertible2}
\end{figure}
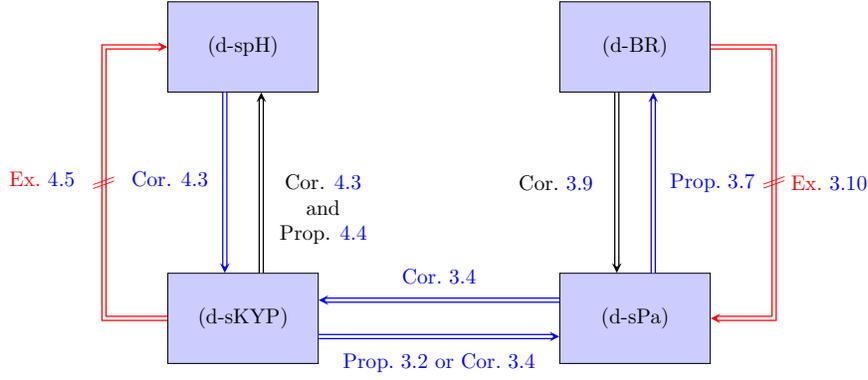

\section{External Cayley transformation of descriptor systems}\label{sec: externalcayley}
In this section, we recall the external Cayley transformation, which was used e.g.\ in \cite{StaW12} for standard state-space systems to relate impedance and scattering passive systems, and we extend this approach to descriptor systems.

To relate a scattering passive system with supply rate $s_{sca}$ depending on the output $y$ and  input $u$ and the supply rate $s_{imp}$ which depends on the output $f$ and input $e$ we consider the following unitary transformation
\begin{align}
\label{ext_cayley}
\begin{bmatrix}
y\\u
\end{bmatrix}=\frac{1}{\sqrt{2}}\begin{bmatrix}
-I_m&I_m\\I_m&I_m
\end{bmatrix}\begin{bmatrix}
f\\e
\end{bmatrix}
\end{align}
which results in 
\begin{align*}
s_{sca}(y,u)&=\begin{bmatrix}
y\\u
\end{bmatrix}^H\begin{bmatrix}
-I_m&0\\0&I_m
\end{bmatrix} \begin{bmatrix}
y\\u
\end{bmatrix}\\&=\frac{1}{2}\begin{bmatrix}
f\\e
\end{bmatrix}^H\begin{bmatrix}
-I_m&I_m\\ \phantom{-}I_m&I_m
\end{bmatrix}^H\begin{bmatrix}
-I_m&0\\0&I_m
\end{bmatrix}\begin{bmatrix}
-I_m&I_m\\ \phantom{-}I_m&I_m
\end{bmatrix}\begin{bmatrix}
f\\e
\end{bmatrix}\\
&=\frac{1}{2}\begin{bmatrix}
f\\e
\end{bmatrix}^H\begin{bmatrix}
-I_m&I_m\\ \phantom{-} I_m&I_m
\end{bmatrix}^H\begin{bmatrix}
I_m&-I_m\\ I_m&\phantom{-}I_m
\end{bmatrix}\begin{bmatrix}
f\\e
\end{bmatrix}\\&= \begin{bmatrix}
f\\e
\end{bmatrix}^H\begin{bmatrix}
0&I_m\\I_m&0
\end{bmatrix}\begin{bmatrix}
f\\e
\end{bmatrix}
\\&=s_{imp}(f,e).
\end{align*}
In \cite{KurS07,StaW12}, the transformation \eqref{ext_cayley} is called \emph{external Cayley transform}, see also \cite{Liv73}, where the name \emph{diagonal transform} was used. It provides a correspondence between scattering passive systems with input $u$ and output $y$ and its corresponding impedance passive system with input $e$ and output $f$. The external Cayley transformation \eqref{ext_cayley} was also applied to transfer functions in Section~\ref{sec:transfer} which provides a relation between positive real and bounded real rational matrix-valued functions.

The external Cayley transformation can be generalized to relate two arbitrary quadratic supply rates of the form \eqref{def:supply}. The only requirement is that the number of positive (hence negative) eigenvalues of the two  Hermitian weight matrices has to coincide. The invertible, but in general not necessarily unitary transformation can in the general case be computed using the spectral decomposition of the weight matrices. In this paper, we do not consider general supply rates further, because we are mainly interested in the connection between impedance and scattering passive systems.

Note that the transformation matrix used in the external Cayley transformation \eqref{ext_cayley} is self-inverse, hence we can also apply the external Cayley transformation to turn an impedance passive system into a scattering passive system. 

In the following proposition, we describe the change of the representation when performing the external Cayley transform of impedance or scattering passive descriptor systems,  see \cite[Proposition 5.1]{StaW12} for a related result for (infinite-dimensional) standard state-space systems and also \cite{ReiS10} for a similar result using only the zero function as a storage function.
\begin{proposition}
\label{prop:impedance_to_scattering}
Suppose that a discrete-time descriptor system with coefficients $(E,A_I,B_I,C_I,D_I)$ is impedance passive and that $I_m+D_I$ is invertible. Then the system 
\[
\begin{bmatrix}
A_S&B_S\\C_S&D_S
\end{bmatrix}=\begin{bmatrix}
A_I-B_I(I_m+D_I)^{-1}C_I&\sqrt{2}B_I(I_m+D_I)^{-1}\\-\sqrt{2}(I_m+D_I)^{-1}C_I&-(I_m+D_I)^{-1}(D_I-I_m)
\end{bmatrix}
\]
is scattering passive and the set of solutions coincides, i.e.\ if $x_k,x_{k+1},u_k,y_k$ fulfill \eqref{discr_DAE} for all $k\geq 0$, then  $x_k,x_{k+1},e_k,f_k$ fulfill \eqref{discr_DAE} for the system given by $(E,A_S,B_S,C_S,D_S)$. Furthermore, every storage function associated with  $(E,A_I,B_I,C_I,D_I)$ is also a storage function of $(E,A_S,B_S,C_S,D_S)$. 
\end{proposition}
\begin{proof}
We express the impedance passive system with inputs $e_k$ and output $f_k$ in \emph{behavior form} 
\begin{align}\nonumber
\begin{bmatrix}Ex_{k+1}\\ 0\end{bmatrix}=\begin{bmatrix}A_I&B_I&0\\C_I&D_I&-I_m\end{bmatrix}\begin{bmatrix}x_k\\ e_k\\ f_k\end{bmatrix}&=\begin{bmatrix}A_I&B_I&0\\C_I&D_I&-I_m\end{bmatrix}\begin{bmatrix}
I_n&0&0\\0&\frac{1}{\sqrt{2}}I_m&\frac{1}{\sqrt{2}}I_m\\0&\frac{1}{\sqrt{2}}I_m&-\frac{1}{\sqrt{2}}I_m
\end{bmatrix}\begin{bmatrix}x_k\\ u_k\\ y_k\end{bmatrix}\\ \label{extcayley}
&=\begin{bmatrix}
A_I&\frac{1}{\sqrt{2}}B_I&\frac{1}{\sqrt{2}}B_I\\C_I&\frac{1}{\sqrt{2}}(D_I-I_m)&\frac{1}{\sqrt{2}}(D_I+I_m)
\end{bmatrix}\begin{bmatrix}x_k\\ u_k\\ y_k\end{bmatrix}.
\end{align}
Since $I_m+D_I$ is invertible, we can multiply the second block row of \eqref{extcayley} with $-\sqrt{2}(I_m+D_I)^{-1}$ without changing the solution set. This gives
\[
\begin{bmatrix}Ex_{k+1}\\ 0\end{bmatrix}=\begin{bmatrix}
A_I&\frac{1}{\sqrt{2}}B_I&\frac{1}{\sqrt{2}}B_I\\-\sqrt{2}(I_m+D_I)^{-1}C_I&-(I_m+D_I)^{-1}(D_I-I_m)&-I_m
\end{bmatrix}\begin{bmatrix}x_k\\ u_k\\ y_k\end{bmatrix}
\]
and we can then use block Gaussian elimination to eliminate the $(1,3)$ block entry of the first row again without changing the solution set and get
\[
\begin{bmatrix}Ex_{k+1}\\ 0\end{bmatrix}=\begin{bmatrix}
A_I-B_I(I_m+D_I)^{-1}C_I&\sqrt{2} B_I(I_m+D_I)^{-1}&0\\-\sqrt{2}(I_m+D_I)^{-1}C_I&-(I_m+D_I)^{-1}(D_I-I_m)&-I_m
\end{bmatrix}\begin{bmatrix}x_k\\ u_k\\ y_k\end{bmatrix},
\]
where we have used that $\frac{1}{\sqrt{2}} (I_m-(I_m+D_I)^{-1}(D_I-I_m))= \sqrt{2}(I_m+D_I)^{-1} $.

This is the desired state space representation of the system with inputs $u_k$ and outputs $y_k$. Since we did not change the solution set, it follows from $s_{sca}(y,u)=s_{imp}(f,e)$ and from (d-iPa) that also (d-sPa) holds with the same storage function. In particular, the system with coefficients $(E,A_S,B_S,C_S,D_S)$ is scattering passive.
\end{proof}

\begin{remark} \label{rem:notinv} {\rm We have seen that for a given impedance passive system, the external Cayley transform leads to a scattering passive system. However, the representation in Proposition~\ref{prop:impedance_to_scattering} can only be derived if $I_m+D_I$ is invertible.
For standard discrete-time state-space systems, it follows from considering the (2,2) entry of the block matrix in (d-iKYP) that $D_I+D_I^H\geq 0$ holds and therefore $I_m+D_I$ is invertible. 

An impedance passive descriptor system \eqref{discr_DAE} with singular $I_m+D$ is given by  $(E,A,B,C,D)=(0,1,-1,1,-1)$. Clearly, $1+D=0$ is singular, but the system fulfills $y_k=x_k-u_k=0$ and is, therefore, impedance passive with storage function $V(x)=0$ for all $x\in\dR$.  
}
\end{remark}
\begin{remark}\label{rem:nicerep}{\rm  
One may wonder why in the discrete-time case the scattering passive case was considered, while the continuous-time case is usually based on the concept of impedance passivity. The reason is that  the characterization in terms of the coefficient matrices is much nicer in the scattering case, while a characterization in the impedance passive case is based on a more complex relationship between the coefficients as demonstrated in the following proposition, see also Section~\ref{sec:cont_to_discr}.} 
\end{remark}
%In the following, we derive impedance passive representations %of scattering passive systems.
\begin{proposition}
\label{prop:scat2imp}
Suppose that a discrete-time descriptor system with coefficients $(E,A_S,B_S,C_S,D_S)$ is scattering passive. If $I_m+D_S$ is invertible then the system 
\begin{align}
\label{eq:impfromscat}
\begin{bmatrix}
A_I&B_I\\C_I&D_I
\end{bmatrix}=\begin{bmatrix}
A_S-B_S(I_m+D_S)^{-1}C_S&\sqrt{2}B_S(I_m+D_S)^{-1}\\-\sqrt{2}(I_m+D_S)^{-1}C_S&-(I_m+D_S)^{-1}(D_S-I_m)
\end{bmatrix}
\end{align}
is impedance passive.

If $\ker(I_m+D_S)\neq\{0\}$ and if there exists  $X=X^H>0$ satisfying \emph{(d-sKYP)} then with $P_{\ker (I_m+D)^\perp}$ denoting a projector on the orthogonal complement ${\ker (I_m+D_S)^\perp}$ of $\ker (I_m+D_S)$, the restricted system 
\begin{align}
\label{red_scat}
\begin{bmatrix}\widehat{A}_S& \widehat{B}_S\\ \widehat{C}_S& \widehat{D}_S\end{bmatrix}:=\begin{bmatrix}A_S&B_S\vert_{\ker (I_m+D_S)^{\perp}}\\ C_S&P_{\ker (I_m+D_S)^{\perp}}D_S\vert_{\ker (I_m+D_S)^{\perp}}\end{bmatrix}
\end{align}
is scattering passive  and  using an external Cayley transformation gives an impedance passive system of the form  \eqref{eq:impfromscat}.
\end{proposition}
%Hence, if $I_m+D$ is invertible we have to consider it on ?some complement space? of $\ker (I_m+D)$ which leads to an invertible operator. 
\begin{proof}
The proof that \eqref{eq:impfromscat} is the state space representation of the external Cayley transformed system follows analogous to the proof of Proposition~\ref{prop:impedance_to_scattering}. 

Next, we show the  inclusion
\begin{align}
    \label{kerID_incl}
\ker (I_m+D_S)\subseteq \ker XB_S\cap\ker C_S^H=\ker B_S\cap\ker C_S^H.
\end{align}
Let $v \in\ker (I_m+D_S)$. Then $D_S v=- v$ holds, and  considering the lower diagonal entry in (d-sKYP) implies that $-B_S^HXB_S v=0$ holds and therefore $v \in\ker XB_S$. Multiplying (d-sKYP) with $\begin{smallbmatrix}
0\\v
\end{smallbmatrix}$ from the right and its conjugate-transpose from the left implies $v \in\ker C_S^H$. This proves \eqref{kerID_incl}.
 
We decompose $u_k=u_k^1+u_k^2$ with $u_k^2\in\ker (I_m+D_S)$ and $u_k^1\in\ker (I_m+D_S)^\perp$ and denote the orthogonal projection onto $\ker (I_m+D_S)$ by $P$.  Then we can write the system equivalently as 
\begin{align*}
Ex_{k+1}=A_Sx_k+B_Su_k^1,\\
y_k^1=C_Sx_k+(I_m-P)D_Su_k^1, \\
y^2_k=PD_Su_k^1-u_k^2,
\end{align*}
where $y^1\in \ker(I_m+D_S)^\perp$ and $y^2\in \ker(I_m+D_S)$ by construction. 

If we choose
\[
u_k^2=\tfrac{1}{2}PD_Su_k^1
\]
then
\[
\|u_k^2\|^2=\|PD_Su_k^1-u_k^2\|^2=\|y_k^2\|^2
\]
and,  together with the orthogonality, we obtain
\[
\|u_k^1\|^2-\|y_k^1\|^2= \|u_k^1\|^2+\|u_k^2\|^2-\|y_k^1\|^2-\|y_k^2\|^2=\|u_k\|^2-\|y_k\|^2.
\]
Hence the reduced system defined in \eqref{red_scat} is scattering passive and has the property that 
\[
\ker(I_{\dim\ker(I_m+D_s)^\perp}+\widehat D_S)=\{0\}
\]
and therefore we can apply the external Cayley transformation to construct an impedance passive system.
\end{proof}

The external Cayley transform of a scattering passive descriptor system leads to an impedance passive descriptor system. However, the standard state-space representation of the resulting impedance passive system can in general only be obtained if $I+D_S$ is invertible as the following example shows. 
\begin{example}
    Consider e.g.\ the system $x_{k+1}=\tfrac12x_k$ with output equation $y_k=-u_k$, i.e.\ $D=-I_m$. Then this system is scattering passive with storage function $V(x)=\|x\|^2$. Applying the external Cayley transform, we obtain a system without output variables and $u_k=0$ for all $k\geq 0$ which is no longer a classical control system. 
\end{example}
We conclude this section with a brief discussion of realizations of positive (bounded) real transfer functions. Recall that if there exists a discrete-time system of the form \eqref{discr_DAE} with $E,A\in\mathbb{K}^{n\times n}$, $B,C^H\in\mathbb{K}^{n\times m}$ and $D\in\mathbb{K}^{m\times m}$ that has the transfer function $\Tc(z)=C(zE-A)^{-1}B+D$, then this is called a realization of $\Tc$. 

By definition, the transfer function of a bounded real system $\Tc$ is bounded, and therefore there exists a realization of index at most one \cite[Section 6]{FreJ04}. It follows from \cite[Theorem 6.3]{FreJ04} that minimal realizations of bounded real transfer functions can always be rewritten equivalently as a standard discrete-time state-space system. 
For systems that have index at most one, this can also be seen easily by \eqref{dtoE}.

This construction of realizations can also be used to perform an index reduction of an impedance passive system as follows: First, perform an external Cayley transform of the system. Then, the transfer function of the transformed system is bounded real. Hence, a minimal realization of this system can be expressed as a scattering passive standard discrete-time state-space system. If required, an external Cayley transform can be applied to obtain an impedance passive standard discrete-time state-space system. This may, however, require a  restriction of the input space, see Proposition~\ref{prop:scat2imp}.

In this section, we have shown how to relate impedance and scattering passive systems via the external Cayley transformations, and that under the extra condition that $I+D$ is invertible, these two passivity conditions are equivalent. Hence, one can define an impedance passive discrete-time pH system via the external Cayley transform and the connection on matrices in Definition~\ref{def:discpH}.

\section{From continuous-time to discrete-time dissipative systems}
\label{sec:cont_to_discr}

In this section, we explore how the passivity properties for a continuous-time system are transferred to a corresponding discretized system.

We show in particular that the so-called internal Cayley transform, corresponding to the trapezoidal rule (or implicit mid-point discretization method for linear systems), will lead to dissipative discrete-time systems which are equivalent to a scattering pH representation.  To proceed, we first recall for the continuous-time case the notions of bounded realness, scattering passivity, and the solution of the corresponding KYP inequalities, see e.g. \cite{CheGH22}.
\begin{itemize}
\setlength{\itemindent}{2em}
\item[\rm (pH)] \label{gl:pH} A continuous-time system of the form \eqref{cont_DAE} is called \emph{port-Hamiltonian} if there exists  $J,R,X\in\K^{n\times n}$, $G,P\in\K^{n\times m}$, and  $S,N\in\K^{m\times m}$ such that 
\begin{align}
\label{def_PH}
\begin{split}
\begin{bmatrix}
A&B\\C&D
\end{bmatrix}&=\begin{bmatrix}(J-R)X&G-P\\(G+P)^HX&S+N\end{bmatrix},\quad X^HE=E^HX\geq 0,\\
	\Gamma &:= \begin{bmatrix}
		J & G \\
		-G^H& N\end{bmatrix}
		= - \Gamma^H,\quad 
 W := \begin{bmatrix}
	X^HRX & X^HP\\
	P^HX & S
\end{bmatrix} =W^H \geq 0.
\end{split}
\end{align}
with the Hamiltonian $\Hc(x) := \frac{1}{2} x^H E^H X x$.
\item[(sPa)] \label{gl:sPa} A continuous-time system of the form \eqref{cont_DAE} is said to be \emph{scattering passive} if it is passive with the supply rate $s_{sca}(u,y)=\|u\|^2-\|y\|^2$. 

\item[(iPa)] \label{gl:iPa} A continuous-time system of the form \eqref{cont_DAE} is said to be \emph{impedance passive} if it is passive with the supply rate $s_{imp}(u,y)=2\Re(y^Hu)$.

\item[(BR)] \label{gl:BR}
A continuous-time system of the form \eqref{cont_DAE} is said to be \emph{bounded real} if the following conditions hold, see \cite{BroLME20}, if 
\begin{itemize}
\setlength{\itemindent}{2em}
\item[\rm (i)] $\mathcal{T}(s)$ analytic for all $\Re(s)>0$,
\item[\rm (ii)] $\Tc(s)$ is real for all $s\in(0,\infty)$,
\item[\rm (iii)] $\|\Tc(s)\|\leq 1$ for all $\Re(s)>0$.
\end{itemize}

\item[(PR)] \label{gl:PR}
A continuous-time system of the form \eqref{cont_DAE} is said to be \emph{positive real} if the following conditions hold, see \cite{BroLME20}, if  $\mathcal{T}(s)$ analytic and if $\Tc(s)+\Tc(s)^H\geq 0$ holds for all $\Re(s)>0$.

\item[] \hspace{-1cm} {\rm (sKYP)} \label{gl:sKYP}
%\cite{GriG15} 
There exist $X=X^H\geq0$  such that the \emph{scattering KYP} is of the form
\[
 \begin{bmatrix}
-A^H X-XA-C^H C&-XB-C^H D\\-B^H X-D^H C&I_m-D^H D
\end{bmatrix}\geq 0.
\]
\item[] \hspace{-1cm} {\rm (iKYP)} \label{gl:iKYP}
There exist $X=X^H\geq0$  such that the \emph{impedance KYP} is of the form
\[
 \begin{bmatrix}
-A^H X-XA&C^H-XB \\C-B^H X&D+D^H
\end{bmatrix}\geq 0.
\]
\end{itemize}
The relations between (BR), (sPa) and (sKYP) for standard state-space systems have been discussed extensively in the literature \cite{BroLME20,GriG15,Sta03}.  Recently, the relations between (PR), (iPa), (iKYP) and continuous-time pH descriptor systems are summarized in \cite{CheGH22} and displayed in Figure~\ref{fig:overview2}.

\begin{figure}
    \centering
    %\scalebox{0.78}{
\begin{tikzpicture}
        \node[block] (a) {(pH)};
        \node[block, below =2cm of a]   (b){(iKYP)};
        %\node[block, right =2cm of b]   (c){(KYP$\mid\Vs$)};
        \node[block, right =2cm of b]   (d){(iPa)};
        \node[block, above =2cm of d]   (e){(PR)};
        %\node[block, above right = 0.2cm and 0.5cm of c ]   (f){finite\\available\\storage};
        
        \draw[->,semithick,double,double equal sign distance,>=stealth, color=blue] ([xshift=-2ex]a.south) -- ([xshift=-2ex]b.north);
        \draw[->,semithick,double,double equal sign distance,>=stealth] ([xshift=2ex]b.north) -- ([xshift=2ex]a.south) node[midway,right = 2 ex]{(a)};
        \draw[->,semithick,double,double equal sign distance,>=stealth, color=blue] ([yshift=-2ex]b.east) -- ([yshift=-2ex]d.west);
        \draw[->,semithick,double,double equal sign distance,>=stealth] ([yshift=2ex]d.west) -- ([yshift=2ex]b.east) node[midway,above = 2 ex]{(b)};
        %\draw[<->,semithick,double,double equal sign distance,>=stealth, color=blue] (c.east) -- (d.west)node[midway,below = 2 ex]{};
        %\draw[<->,semithick,double,double equal sign distance,>=stealth,color=blue] (c.east) -| (f.south)node[midway,above right = 2 ex and 2ex]{};
         %\draw[-,semithick, white,line width=1.4pt, shorten >= 7pt] ([xshift=2ex]c.east) -- (d.west);
        %\draw[-,semithick, white,line width=1.4pt, shorten >= 7pt] ([xshift=2ex]c.east) -| (f.south);
        %\draw[->,semithick,double,double equal sign distance,>=stealth,negated, color=red] ([xshift=-2ex]e.south) -- ([xshift=-2ex]d.north)node[midway,left = 2 ex]{Ex.~\ref{ex:obsv!notcontr}~~};
        \draw[->,semithick,double,double equal sign distance,>=stealth, color=blue] ([xshift=2ex]d.north) -- ([xshift=2ex]e.south) ;
        \draw[->,semithick,double,double equal sign distance,>=stealth, color=black] ([xshift=-2ex]e.south) -- ([xshift=-2ex]d.north) node[midway, left = 2ex] {(c)};
        \draw[->,semithick,double,double equal sign distance,>=stealth] ([yshift=2ex]e.west) -- ([yshift=2ex]a.east)node[midway,above = 2 ex]{minimal};
        %\draw[->,semithick,double,double equal sign distance,>=stealth, color=black] ([yshift=-2ex]e.west) -| (c.north)node[midway,above right= 0.5ex and 3ex]{beh. controllable~~};
        %\draw[->,semithick,double,double equal sign distance,>=stealth] (d.south) |-  ++(0,-10pt) coordinate[yshift=-1.7cm,](r){} -| (b.south)node[midway,below right= 2 ex and 6cm]{(2)};
        %\draw[->,semithick,double,double equal sign distance,>=stealth] (b.west) --  ++(-10pt,0) coordinate[yshift=-1.7cm,](r){} |- (a.west)node[midway,below left= 1.5cm and 1ex]{(1)};
        \end{tikzpicture}
        %}
       \caption{Overview of the relationship between (pH), (iKYP), (iPa) and (PR) for continuous-time descriptor systems with $(E,A)$ regular and $E\neq0$. The implications with additional assumptions are colored black and the ones without in blue. 
       \\
       (a) $\ker X \subseteq \ker C\cap\ker A$, 
       (b) index at most one, 
       (c) behaviorally controllable.}
    \label{fig:overview2}
\end{figure}
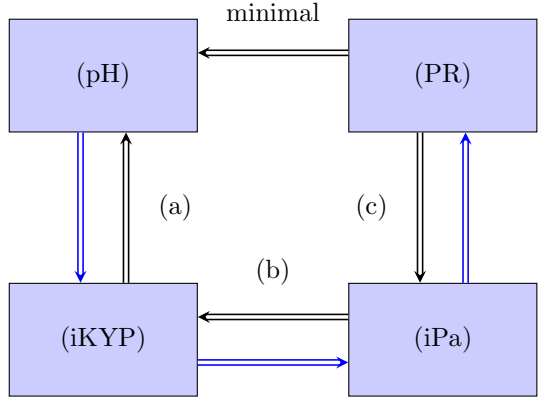

Next, the well-known \emph{Tustin discretization} (also called trapezoidal rule or implicit midpoint rule for linear systems) which was described for standard state-space systems in \cite[Section 3]{FraPW98} is applied to continuous-time descriptor systems. The following considerations are based on \cite[Section 5.1]{KurS07} for standard state-space systems and the more general \emph{input-state-output systems}, see also \cite{Meh96} for a detailed analysis. 

For a discretization step-size $h\in(0,\infty)$ consider the equidistant time grid $\{t_k\}_{k\geq 0}$ of $[0,\infty)$ with $t_k=kh$. Then for the continuous-time  descriptor system \eqref{cont_DAE} in $(t_k,t_{k+1})$ we obtain
\[
Ex(t_{k+1})-Ex(t_k)=\int_{t_k}^{t_{k+1}} \tfrac{d}{dt} Ex(\tau)\, d\tau=\int_{t_k}^{t_{k+1}}Ax(\tau)+Bu(\tau)\, d\tau.
\]
Approximating the integral by the trapezoidal rule results in 
\[
Ex(t_{k+1})-Ex(t_k)=\frac{h}{2}(Ax(t_{k+1})+Bu(t_{k+1})+Ax(t_{k})+Bu(t_k)).
\]
If $(E,A)$ is regular, then for sufficiently small $h>0$ one has $\tfrac{h}{2}\in\rho(E,A)$. This allows us to multiply the last equation with the resolvent from the left and obtain
\[
x(t_{k+1})=(\tfrac{2}{h}E-A)^{-1}(\tfrac{2}{h}E + A)x(t_k) +(\tfrac{2}{h}E-A)^{-1}B(u(t_{k+1})+u(t_k)).
\]
This can be used to define a \emph{discrete-time} system with state sequence $\{x_k\}$ and input sequence $\{u_k\}$, where $x_k\approx x(t_k)$ for all $k\geq 0$ and, using $\alpha:=\tfrac{2}{h}$, 
\begin{align}
\label{discrete_without_y}
\begin{split}
x_{k+1}&=(\alpha E-A)^{-1}(\alpha E+A)x_k+\sqrt{2\alpha}(\alpha E-A)^{-1}Bu_k, \\ u_k&:=\frac{u(t_{k+1})+u(t_k)}{\sqrt{2\alpha}}.
\end{split}
\end{align}

For  continuous-time standard state-space systems with coefficients $(I_n,A,B,C,D)$, the discrete-time system \eqref{discrete_without_y} can be extended by an additional output equation in a passivity preserving way. The resulting discrete-time system is often called \emph{internal Cayley transformation} of the continuous-time system and is for some $\alpha\in \dC_+\cap\rho(A)$   given by 
\begin{align}
\label{def:int_cayley}
\begin{bmatrix}
\textbf{A}&\textbf{B}\\\textbf{C}&\textbf{D}
\end{bmatrix}:=\begin{bmatrix}
(\alpha I_n -A)^{-1}(\overline{\alpha}I_n+A)&\sqrt{2\Re(\alpha)}(\alpha I_n-A)^{-1}B \\ \sqrt{2\Re(\alpha)}C(\alpha I_n-A)^{-1}&\Tc(\alpha)
\end{bmatrix},
\end{align}
where $\Tc(\alpha)=C(\alpha I_n-A)^{-1}B+D$ is the transfer function of the system. An overview of results for the internal Cayley transform applied to (possibly infinite dimensional) standard state-space systems is given e.g.\ in \cite[Section 4]{StaW12}. Furthermore, it is easy to see that the internal Cayley transform preserves the controllability and observability of a given continuous-time system, i.e.\ if $(A,B)$ is controllable (resp.\ $(A,C)$ observable) then $(\mathbf{A},\mathbf{B})$ is controllable (resp.\ $(\mathbf{A},\mathbf{C})$ observable).

Furthermore, the transfer function of the resulting discrete-time system is given by 
\[
\textbf{C}(zI_n-\textbf{A})^{-1}\textbf{B}+\textbf{D}=\Tc(\tfrac{\alpha z-\overline{\alpha}}{z+1}),\quad \vert z\vert>1.
\]
This follows from
\begin{align*}
&~~~~\textbf{C}(zI_n-\textbf{A})^{-1}\textbf{B}+\textbf{D}\\&=2\Re(\alpha) C(\alpha I_n-A)^{-1}(z I_n-(\alpha I_n-A)^{-1}(\overline{\alpha}I_n+A))^{-1}(\alpha I_n-A)^{-1}B+\mathcal{T}(\alpha)\\&=2\Re(\alpha) C(\alpha I_n-A)^{-1}(z (\alpha I_n-A)-(\overline{\alpha}I_n+A))^{-1}B+\mathcal{T}(\alpha)\\
&=\Tc(\tfrac{\alpha z-\overline{\alpha}}{z+1}),
\end{align*}
where in the last step we have used  that 
\begin{align*}
&~~~~\frac{2\Re(\alpha)}{z+1}(\alpha I_n-A)^{-1}\left(\frac{z\alpha-\overline{\alpha}}{z+1} I_n-A\right)^{-1}\\&=\frac{2\Re(\alpha)}{z+1}\left(\alpha-\frac{z\alpha-\overline{\alpha}}{z+1}\right)^{-1}\left(\left(\frac{z\alpha-\overline{\alpha}}{z+1}I_n-A\right)^{-1}-(\alpha I_n-A)^{-1}\right).
\end{align*}
Hence, the transfer function of the discretized system fulfills (d-PR) (resp.\ (d-BR)) if and only if the transfer functions of the continuous-time system fulfill (PR) (resp.\ (BR)). 

%\VMcomment{Difficulties between %continuous and discrete-time were %studied in \cite{Meh96}}

The following result was obtained for the special case of systems that have $X=I_n$ as a solution to (sKYP) in \cite[Proposition 4.3]{StaW12}. 
\begin{proposition}\label{prop:cont_to_discrete}
Consider a standard continuous-time state-space systems with coefficients $(I_n,A,B,C,D)$ that has a solution $0\leq X=X^H$ to \emph{(sKYP)} (resp.\ \emph{(iKYP)}) and let $\alpha\in\rho(E,A)\cap\dC_+$. Then the discrete-time system $(I_n,\mathbf{A},\mathbf{B},\mathbf{C},\mathbf{D})$ given by \eqref{def:int_cayley} is scattering (resp.\ impedance) passive with storage function $V(x)=\tfrac12x^HXx$. 
\end{proposition}
\begin{proof}
The idea of the proof is based on \cite{KurS07} where  a congruence transformation with
\begin{align}
    \label{congruence}
T:=\begin{bmatrix}
  \sqrt{2\Re(\alpha)}  (\alpha I_n-A)^{-1}&  (\alpha I_n-A)^{-1}B\\0&I_m
\end{bmatrix}
\end{align}
is used.
As a first step, we show the following equality
\begin{align}
\label{eq:without_costs}
T^H
 \begin{bmatrix}
-A^H X-X A&-XB\\-B^H X&0
\end{bmatrix}T=\begin{bmatrix}
-\mathbf{A}^HX\mathbf{A}+X&-\mathbf{A}^HX\mathbf{B}\\-\mathbf{B}^H X\mathbf{A}&-\mathbf{B}^HX\mathbf{B}
\end{bmatrix}.
\end{align}
The left hand side of \eqref{eq:without_costs} can be rewritten as
\begin{align}
\label{eq:TH_left}
 T^H\begin{smallbmatrix}
-\sqrt{2\Re(\alpha)}(A^H X+XA) (\alpha I_n-A)^{-1}&-(A^H X+XA) (\alpha I_n-A)^{-1}B-XB\\\sqrt{2\Re(\alpha)}B^HX(\alpha I_n-A)^{-1}&-B^HX(\alpha I_n-A)^{-1}B
\end{smallbmatrix}.
\end{align}
The result of the multiplication with $T^H$ from the left in \eqref{eq:TH_left} is considered for each block entry separately. The (1,1) entry of the resulting block matrix in \eqref{eq:TH_left} is given by 
\begin{align*}
&~~~~2\Re(\alpha) (\alpha I_n-A)^{-H}( -XA -A^HX)(\alpha I_n-A)^{-1}\\
&=(\alpha I_n-A)^{-H}( -\alpha XA -\overline{\alpha}A^HX-\overline{\alpha}XA-\alpha A^HX)(\alpha I_n-A)^{-1}\\
&=(\alpha I_n-A)^{-H}(-(\alpha I_n+A^H)X(\overline{\alpha}I_n+A)+(\alpha I_n-A)^{H}X(\alpha I_n-A))(\alpha I_n-A)^{-1}\\
&=-\mathbf{A}^HX\mathbf{A}+X,
\end{align*}
which proves \eqref{eq:without_costs} for the (1,1) entries.

Furthermore, the (1,2) entry of the resulting block matrix in \eqref{eq:TH_left} equals
\begin{align*}
&~~~~\sqrt{2\Re(\alpha)}(\alpha I_n-A)^{-H}((-A^H X-XA)(\alpha I_n-A)^{-1}B-XB)\\
&=\sqrt{2\Re(\alpha)}(\alpha I_n-A)^{-H}((-A^H X-XA-X(\alpha I_n-A))(\alpha I_n-A)^{-1}B)\\
&=-(\alpha I_n-A)^{-H}(\overline{\alpha}I_n+A)^H X \mathbf{B}\\
&=-\mathbf{A}^HX\mathbf{B}.
\end{align*}
Hence, \eqref{eq:without_costs} holds for the (1,2) entries of the block matrices. Since both matrices in \eqref{eq:without_costs} are Hermitian, the (2,1) entries coincide as well. It remains to show equality of the (2,2) entries 
\begin{align*}
&B^H(\alpha I_n-A)^{-H}(-(A^H X+XA) (\alpha I_n-A)^{-1}B-XB)-B^HX(\alpha I_n-A)^{-1}B\\&=B^H(\alpha I_n-A)^{-H}(-(A^H X+XA+X(\alpha I_n-A)) (\alpha I_n-A)^{-1}B)\\ & ~~~~-B^HX(\alpha I_n-A)^{-1}B \\
&=B^H(\alpha I_n-A)^{-H}(-(A^H X+X\alpha) (\alpha I_n-A)^{-1}B)\\&~~~~-B^HX(\alpha I_n-A)^{-1}B\\
&=-2\Re(\alpha) B^H(\alpha I_n-A)^{-H}X(\alpha I_n-A)^{-1}B
\\
&=-\mathbf{B}^HX\mathbf{B}.
\end{align*}
In summary, this proves \eqref{eq:without_costs}. In the second step we use the transformation \eqref{congruence} to obtain 
\begin{align*}
&~~~~T^H\begin{bmatrix}
    0& C^H\\ C & D+D^H
\end{bmatrix}T\\ &=\begin{bmatrix}
  \sqrt{2\Re(\alpha)}  (\alpha I_n-A)^{-1}&  (\alpha I_n-A)^{-1}B\\0&I_m
\end{bmatrix}^H\begin{bmatrix}
    0&C^H\\\sqrt{2\Re(\alpha)}C(\alpha I_n-A)^{-1}&\Tc(\alpha)+D^H
\end{bmatrix}\\
&=\begin{bmatrix}
    0&\mathbf{C}^H\\\mathbf{C}&\mathbf{D}+\mathbf{D}^H
\end{bmatrix}.
\end{align*}
Since the congruence transformation with $T$ preserves positive semidefiniteness, we conclude that if (iKYP) holds for $X=X^H\geq 0$, then $X$ is also a solution to (d-iKYP) for the system $(I_n,\mathbf{A},\mathbf{B},\mathbf{C},\mathbf{D})$. 

Analogously, we use the transformation \eqref{congruence} to obtain 
\begin{align*}
&~~~~T^H\begin{bmatrix}
    -C^HC& -C^HD\\ -D^HC & I_m-D^HD
\end{bmatrix}T\\ &=\begin{smallbmatrix}
  \sqrt{2\Re(\alpha)}  (\alpha I_n-A)^{-1}&  (\alpha I_n-A)^{-1}B\\0&I_m
\end{smallbmatrix}^H\begin{smallbmatrix}
    -\sqrt{2\Re(\alpha)}C^HC(\alpha I_n-A)^{-1}&-C^H\Tc(\alpha)\\ -\sqrt{2\Re(\alpha)}D^HC(\alpha I_n-A)^{-1}&-D^H \Tc(\alpha)+I_n
\end{smallbmatrix}\\
&=\begin{bmatrix}
    -\mathbf{C}^H\mathbf{C}&-\mathbf{C}^H\mathbf{D}\\ -\mathbf{D}^H\mathbf{C}&I_m-\mathbf{D}^H\mathbf{D}
\end{bmatrix}.
\end{align*}
Therefore, if (sKYP) holds for $X=X^H\geq 0$ then $X$ fulfills (d-sKYP) for the system $(I_n,\mathbf{A},\mathbf{B},\mathbf{C},\mathbf{D})$.
\end{proof}

Finally, we show how discrete-time scattering pH systems as in Definition~\ref{def:discpH} can be obtained from time-discretizations  \eqref{def:int_cayley} applied to standard continuous-time pH systems which are given by \eqref{def_PH} with $E=I_n$. Proposition~\ref{prop:cont_to_discrete} shows that the discretization \eqref{def:int_cayley} preserves the impedance passivity of the continuous-time pH system \eqref{def_PH} and the  storage function $V(x)=\tfrac12x^HXx$.

Then it was shown in Proposition~\ref{prop:impedance_to_scattering} that we obtain a scattering passive standard state-space discrete-time system 
\[
\begin{bmatrix}
\mathbf{A}_S&\mathbf{B}_S\\\mathbf{C}_S& \mathbf{D}_S\end{bmatrix}
:=\begin{bmatrix}
\mathbf{A}-\mathbf{B}(I_m+\mathbf{D})^{-1}\mathbf{C}&\sqrt{2}\mathbf{B}(I_m+\mathbf{D})^{-1}\\-\sqrt{2}(I_m+\mathbf{D})^{-1}\mathbf{C}&-(I_m+\mathbf{D})^{-1}(\mathbf{D}-I_m)
\end{bmatrix}
\]
which has the same storage function $V(x)=\tfrac12x^HXx$ as the continuous-time pH system \eqref{def_PH}, which coincides with the Hamiltonian of the pH system. 

Hence, if we assume that $X=X^H>0$ for the continuous-time pH system \eqref{def_PH} then we obtain an equivalent discrete-time scattering pH system in the sense of Definition~\ref{def:discpH} with the following scattering pH representation 
\[
z_{k+1}=X^{\tfrac{1}{2}}\mathbf{A}_SX^{-\tfrac{1}{2}}z_k+X^{\tfrac{1}{2}}\mathbf{B}_Su_k,\quad y_k=\mathbf{C}_SX^{-\tfrac{1}{2}}z_k+\mathbf{D}_Su_k,\quad k\geq 0.
\]

%We also have the immediate consequence for the relation between impedance and scattering passive discrete-time descriptor systems. 
%Consider a completely causal impedance passive discrete-time descriptor system of the form \eqref{discr_DAE} with coefficients $(E,A,B,C,D)$ and the associated standard state-space discrete-time system \eqref{eq:indexone_ODE} with coefficients $(I_n,\Ac,\Bc,\Cc,\Dc)$ for which there exists a solution  $0<X=X^H\in\K^{n\times n}$ of (d-iKYP). 

In this section, we have shown that the discretization-based approach to move from a continuous-time to a discrete-time via the implicit mid-point rule (or internal Cayley transformation) preserves the relationship between the different characterizations of dissipativity.
%In the next section, we discuss the geometric formulation via Dirac structures.

\section{Conclusion}

In this paper, we have considered discrete-time descriptor systems and studied different passivity concepts such as scattering and impedance passivity.
The characterizations of these concepts were analyzed via the concepts of positive (bounded) realness as well as the solvability of Kalman-Yakubovich-Popov inequalities. In addition, we have derived equivalence conditions under further assumptions. We also introduced a definition of discrete-time dissipative port-Hamiltonian systems that are purely based on the coefficient matrices and analyzed their properties.  It was shown in the paper how this new definition relates to classical definitions that are derived via the discretization of continuous-time port-Hamiltonian (descriptor) systems.

\section*{Glossary}
\begin{center}
\setlength{\tabcolsep}{2pt}
\scalebox{0.85}{
\begin{tabular}{||l |l | l||} 
 \hline
 Abbreviation & Full name & \begin{tabular}{@{}c@{}}Reference  \\ in the text\end{tabular}  \\ [0.5ex] 
 \hline\hline
 C1 & behaviorally controllable & p.\ \pageref{gl:C1}  \\ 
 \hline
 C1 and C2 & strongly controllable &  p.\ \pageref{gl:str_contrl}  \\
 \hline
 O1 & behaviorally observable & p.\ \pageref{gl:O1}  \\ 
 \hline
 O1 and O2 & strongly observable & p.\ \pageref{gl:str_obsv}  \\
 \hline
 d-sPA & discrete-time scattering passive & eq.\ \eqref{eq:dsPA}, p.\ \pageref{eq:dsPA}  \\
 \hline
  d-iPA & discrete-time
  impedance passive & eq.\  \eqref{eq:diPA}, p.\ \pageref{eq:diPA}   \\
 \hline
 d-iKYP & discrete-time KYP inequality for impedance supply rate & eq.\ \eqref{eq:diKYP}, p.\ \pageref{eq:diKYP}  \\
 \hline
  d-sKYP & discrete-time KYP inequality for scattering supply rate & eq.\ \eqref{eq:dsKYP}, p.\ \pageref{eq:dsKYP}   \\
 \hline
   d-PR & discrete-time positive real & Def.\  \ref{def:bdposreal}, p.\ \pageref{def:bdposreal}  \\
 \hline
    d-BR & discrete-time bounded real & Def.\  \ref{def:bdposreal}, p.\ \pageref{def:bdposreal}  \\
 \hline
     d-spH & discrete-time port Hamiltonian with scattering supply rate & Def.\  \ref{def:discpH}, p.\ \pageref{def:discpH} \\
 \hline
      pH & continuous-time port Hamiltonian & p.\ \pageref{gl:pH}  \\
 \hline
      PR & continuous-time positive real transfer function & p.\ \pageref{gl:PR}  \\
 \hline
      iPa & continuous-time impedance  supply rate & p.\ \pageref{gl:iPa}  \\
 \hline
      iKYP & continuous-time KYP inequality for impedance supply rate & p.\ \pageref{gl:iKYP}   \\  
 \hline
      BR & continuous-time bounded real transfer function & p.\ \pageref{gl:BR} \\
 \hline
      sPa & continuous-time scattering supply rate & p.\ \pageref{gl:sPa}  \\
 \hline
      sKYP & continuous-time KYP inequality for scattering supply rate & p.\ \pageref{gl:sKYP}  \\ %[1ex]
 \hline
\end{tabular}
}
\end{center}

\section*{Acknowledgments}
The work of K.~Cherifi has been supported by ProFIT (co-financed by the Europäischen Fonds für regionale Entwicklung (EFRE)) within the WvSC project: EA 2.0 - Elektrische Antriebstechnik (project No. 10167552).

The work of D.~Hinsen and V.~Mehrmann has been supported by the Deutsche Forschungsgemeinschaft (DFG, German Research Foundation) CRC 910 \emph{Control of self-organizing nonlinear systems: Theoretical methods and concepts of application}: Project No.~163436311 and by Bundesministerium für Bildung und Forschung (BMBF)  EKSSE: Energieeffiziente Koordination und Steuerung des Schienenverkehrs in Echtzeit (grant no. 05M22KTB).

The work of H.~Gernandt has been supported by the Deutsche Forschungsgemeinschaft (DFG, German Research Foundation) within the Priority Programme 1984 ``Hybrid and multimodal energy systems'' (Project No.~ 361092219) and the Wenner-Gren Foundation. 

Furthermore, we thank the anonymous  referees for their careful reading and valuable comments which improved the overall quality of the manuscript.

\bibliographystyle{plain}
\bibliography{sn-bibliography}

\begin{appendix}

    \section{Known results for discrete-time descriptor systems}
\subsection{Solution formula}
\label{sec:solvability}
For regular continuous-time descriptor systems of the form \eqref{cont_DAE} with commuting coefficients  $AE=EA$ a solution formula based on the \emph{Drazin inverse} has been presented in \cite[Section 2.2]{KunM06}. Note that the commutativity can always be achieved by redefining $\widehat E:=(\lambda E-A)^{-1}E$ and $\widehat A:=(\lambda E-A)^{-1}A$ for some $\lambda\in\rho(E,A)$. Let $E$ have nilpotency index $\nu$, % be the  $E$, 
%of the DAE $(E,I_n)$. 
then, see e.g.\ \cite[Theorem 2.19]{KunM06}, the Drazin inverse $E^D$ is the unique matrix satisfying 
\begin{align}
    \label{Draz_def}
EE^D=E^DE,\quad E^DEE^D=E^D,\quad E^DE^{\nu+1}=E^\nu.
\end{align}

An explicit solution formula for discrete-time systems of the form \eqref{discr_DAE} on the basis of the Drazin inverse is presented in \cite[Theorem 4.1]{Bru09}.
\begin{proposition}
\label{prop:voc_formula}
Let $E,A\in\dC^{n\times n}$ with $EA=AE$ such that $(E,A)$ is regular with index $\nu$ and let $f_k\in\dC^n$, $k\geq 0$. Then the solution of $Ex_{k+1}=Ax_k+f_k$, $k\geq 0$ is for some $v\in\dC^n$ given by 
\begin{align}
    \label{eq:brüll_formula}
x_k=(E^DA)^kE^DEv+\sum_{j=0}^{k-1}(E^DA)^{k-j-1}E^Df_j-(I-E^DE)\sum_{i=0}^{\nu-1}(A^DE)^iA^Df_{k+i}.
\end{align}
\end{proposition}
This formula implies that for DAEs with index $\nu \geq 2$, the state $x_k$ at time $k$ may depend on future inputs $f_{k+1},\ldots,f_{k+\nu-1}$ leading to \emph{non-causal} systems. The formula also shows that initial values and input functions may be restricted so that there exists a $v$ in \eqref{eq:brüll_formula} satisfying
\begin{equation}\label{concon}
(E^DA)^kE^DEv=x_0+(I-E^DE)\sum_{i=0}^{\nu-1}(A^DE)^iA^Df_{i}.
\end{equation}
Such combinations of initial values $x_0$ and inhomogeneities $f_0,\ldots,f_{\nu-1}$ are called \emph{consistent}.

Apart from using Drazin inverses, an explicit solution formula based on sequences of subspaces, which are called \emph{Wong sequences} is given for continuous-time systems in \cite{BerIT12}. This approach can also be extended to discrete-time systems. Alternatively, the solutions can be derived from decoupled forms such as the Weierstra\ss\ form \eqref{eq:weier}.

In the following we use Proposition~\ref{prop:voc_formula} to prove Proposition~\ref{prop:causal}.\\[1ex]
{\bf Proof of Proposition \ref{prop:causal}:}
If \eqref{dae_discr} has index $\nu \leq 1$, then we obtain immediately from \eqref{eq:brüll_formula} that for all $i\geq 1$ $x_k$ does not depend on  future inhomogeneities $f_{k+i}$ and hence is completely causal.

Conversely, assume that the system is completely causal. Then we may assume without restriction that the coefficients $E$ and $A$ are already given in Weierstra\ss\ form. In this case, we have
\[
E=\begin{bmatrix} I&0 & 0\\0 & I & 0\\ 0& 0&N\end{bmatrix},\quad A=\begin{bmatrix} A_0&0&0\\0&A_1&0\\0&0& I\end{bmatrix},\quad E^D=\begin{bmatrix} I&0 & 0\\0 & I & 0\\ 0& 0&0\end{bmatrix},\quad A^D=\begin{bmatrix} 0&0&0\\0&A_1^{-1}&0\\0&0& I\end{bmatrix},
\]
where we have separated the nilpotent part $A_0$ in the Jordan form of $A$ form the invertible part $A_1$. 
The expressions for the Drazin inverse follow from the fact that the above matrices fulfill \eqref{Draz_def} and therefore uniquely determine the Drazin inverse. Suppose that $\nu>1$ then a short calculation using \eqref{eq:brüll_formula} shows that 

With the solution formula \eqref{eq:brüll_formula} we then get
\begin{align}
(I-E^DE)\sum_{i=0}^{\nu-1}(A^DE)^iA^Df_{k+i}
=\sum_{i=0}^{\nu-1}\begin{bmatrix} 0&0&0\\0&0&0\\0&0& N^{i} \end{bmatrix}f_{k+i}\neq 0. 
\label{eq:brüll_wcf}
\end{align}
Hence, $\nu\leq 1$ must hold because otherwise $x_k$ would depend on future inhomogeneities $f_{k+1}$.

For the proof of the second statement, we consider special inputs $f_k=Bu_k$ for a sequence $(u_k)_{k\geq 0}$ and $y_k=Cx_k$ for all $k\geq 0$. Assume that $SET$ and $SAT$ are in Weierstra\ss\ form \eqref{eq:weier} for some invertible $S,T\in\dC^{n\times n}$. Let $SB=\begin{bmatrix} B_1\\ B_2 \end{bmatrix}$ and $CT=\begin{bmatrix} C_1 & C_2   
\end{bmatrix}$, then 
\begin{align*}
D+C(zE-A)^{-1}B&=CT(zSET-SAT)^{-1}SB\\&=D+\begin{bmatrix} C_1 & C_2   
\end{bmatrix}\begin{bmatrix} (zI_r-A_f)^{-1}&0\\ 0& (zN-I_{n-r})^{-1}\end{bmatrix}\begin{bmatrix}B_1\\ B_2\end{bmatrix}.
\end{align*}
To study the properness of the considered rational functions, we consider them restricted to the set $U_R:=\{z\in\dC ~:~ \vert z\vert >R\}$ for some $R>0$. Then $z\mapsto D+C(zE-A)^{-1}B$ is proper if and only if $z\mapsto C_1(zI_r-A_f)^{-1}B_1$ and $z\mapsto C_2(zN-I_{n-r})^{-1}B_2$ are bounded on $U_R$ for some $R>0$.

If we choose $R>\|A_f\|$ then, applying the Neumann series, there exists $M>0$ satisfying for all $z\in U_R$  
\[
\|(z I_r-A_f)^{-1}\|=\vert z\vert^{-1}\|(I_r-z^{-1}A_f)^{-1}\|\leq \vert z\vert^{-1}\sum_{k=0}^{\infty}\frac{\|A_f\|^k}{\vert z\vert^k} \leq M\vert z \vert^{-1}.
\]
Hence, $z\mapsto D+C(zE-A)^{-1}B$ is proper if and only if $z\mapsto C_2(zN-I_{n-r})^{-1}B_2$ is bounded on $U_R$ for some $R>0$.
Moreover, since $N$ is nilpotent, the Neumann series yields %
\[
C_2(zN-I_{n-r})^{-1}B_2=-C_2\sum_{i=0}^{\nu-1}(zN)^iB_2
\]
which is bounded if and only if $C_2N^iB_2=0$ for all $i=1,\ldots,\nu-1$. If we consider \eqref{eq:brüll_wcf} then we see that this is equivalent to $y_k$ not depending on future inputs $u_{k+i}$ for all $i\geq 1$.
%\end{proof}

\subsection{Stability of discrete-time descriptor systems}
\label{sec:stable}
In this subsection, we recall the classic stability notions for discrete-time standard state-space systems, see e.g.\ \cite[Chapter 1]{LaS76}, \cite[Chapter 4]{HeiRV21}, or \cite{AchAM23}.
\begin{definition}\label{def:stab}
Let $A\in\K^{n\times n}$. Then the discrete-time system $x_{k+1} = A x_k$,  $k\geq 0$, is called \emph{stable} if for all $x_0\in\K^n$ there exists a constant $M>0$ such that 
\[
\|x_k\|=\|A^kx_0\|\leq M\|x_0\|,\quad \text{for all $k\geq 0$}.
\]
The system is called \emph{asymptotically stable} if $\lim\limits_{k\rightarrow\infty}x_k=0$ holds for all $x_0\in\K^n$.
\end{definition}
Asymptotic stability is then characterized as follows.
\begin{proposition}
\label{prop:stable}
Let $A\in\K^{n\times n}$. Then for the discrete-time system $x_{k+1}=Ax_k$, $k\geq 0$  the following assertions are equivalent:
\begin{itemize}
\setlength{\itemindent}{2em}
    \item[\rm (i)] the system is stable;
    \item[\rm (ii)] all eigenvalues $\lambda$ of $A$ satisfy $\vert\lambda\vert\leq 1$ and if $\vert\lambda\vert=1$ then $\lambda$ is semi-simple;
    \item[\rm (iii)] there exists a positive definite matrix $X=X^H\in\K^{n\times n}$ such that the  \emph{Lyapunov inequality} 
    \begin{equation}\label{eqn:defstable}
        -A^H X A + X \geq 0
    \end{equation}
is satisfied.
\end{itemize}
\end{proposition}
The same stability definition can be used for discrete-time descriptor systems, see also \cite{DuLM13,MehU23} for the continuous-time case.
\begin{definition}
Let $(E,A)$ with $E,A\in\K^{n\times n}$ be a regular pair. Then the system \eqref{dae_discr} is called \emph{stable} if there exists $M>0$ such that for all $x_0\in\K^n$ for which a solution exists one has $\|x_k\|\leq M\|x_0\|$ for all $k\geq 0$. The system is called \emph{asymptotically stable} if 
$\lim\limits_{k\rightarrow\infty}x_k=0$ holds for all $x_0\in\K^n$.
\end{definition}
If the pair $(E,A)$ is in Weierstra\ss\ canonical form \eqref{eq:weier}, then the stability of \eqref{dae_discr} is characterized by that of the standard state-space system with the matrix $A_f$. In particular, the stability definition of \eqref{dae_discr} does not restrict the index and hence still allows for non-causal systems. However, since stability characterizes the properties of solutions under perturbations of the initial value,  one can see from the solution formula, that perturbations in the initial value  may lead to inconsistent initial conditions, so that the system may not have a solution. 

In the following, we characterize stability and complete causality  of \eqref{discr_DAE} in terms of the existence of special solutions $X=X^H$ to the generalized discrete-time Lyapunov inequality 
\begin{align}
\label{eq:gen_lyap}
-A^HXA+E^HXE\geq 0.
\end{align}

\begin{proposition}
\label{prop:lyapforcausalstable}
Let $(E,A)$ with $E,A\in\K^{n\times n}$ form a regular pair. Then the system~\eqref{dae_discr} is completely causal and stable if and only if \eqref{eq:gen_lyap} has a  solution $0\leq X=X^H\in\K^{n\times n}$ which satisfies $x^HXx>0$ for all $x\in\ima E\setminus\{0\}$.
\end{proposition}
\begin{proof}
Observe that the solvability of \eqref{eq:gen_lyap} as well as the complete causality and stability are invariant under equivalence transformations of the system. Hence we may assume that $(E,A)$ is in Weierstra\ss\ form \eqref{eq:weier}. If the pair is completely causal and stable, then the index of $(E,A)$ is at most one by Proposition~\ref{prop:causal}. Hence the generalized Lyapunov inequality \eqref{eq:gen_lyap} is given by 
\[
-\begin{bmatrix}
A_f&0\\0&I_{n-r}
\end{bmatrix}^H X \begin{bmatrix}
A_f&0\\0&I_{n-r}
\end{bmatrix}+\begin{bmatrix}
I_r&0\\0&0
\end{bmatrix}^H X\begin{bmatrix}
I_r&0\\0&0
\end{bmatrix}\geq 0,\quad X=\begin{bmatrix}
X_{11}&X_{12}\\ X_{21}& X_{22}
\end{bmatrix}. %X_{11}>0.
\]
Forming the product we get
\[
\begin{bmatrix}
A_f&0\\0&I_{n-r}
\end{bmatrix}^H X \begin{bmatrix}
A_f&0\\0&I_{n-r}
\end{bmatrix}=
\begin{bmatrix}
A_f^HX_{11}A_f&A_f^HX_{12}\\X_{21}A_f& X_{22}
\end{bmatrix}.
\]
Since \eqref{dae_discr} is stable, system $x_{k+1}=A_f x_k$  is stable and hence, by Proposition~\ref{prop:stable} there exists positive definite $0< X_{11}^H=X_{11}\in\K^{r\times r}$ such that $-A_f^HX_{11}A_f+X_{11}\geq 0$. If we choose $X_{12}=X_{21}^H=0$ and $X_{22}=0$, then $X$ solves the generalized Lyapunov inequaliy \eqref{eq:gen_lyap}. Conversely, assume that \eqref{eq:gen_lyap} has a solution $X=X^H=\begin{bmatrix}
X_{11}&  X_{12}\\ X_{21}& X_{22}
\end{bmatrix}\geq 0$ with positive definite $X_{11}>0$. Considering the upper diagonal block of \eqref{eq:gen_lyap} implies that
\[
-A_f^HX_{11}A_f+X_{11}\geq 0,%\quad %X_{11}>0,
\]
which means that $X_{11}>0$ solves the discrete-time Lyapunov inequality. Using Proposition~\ref{prop:stable} we conclude that the standard state-space system with the coefficient matrix $A_f$ is stable. Hence, by 
definition,~\eqref{discr_DAE} is stable. 

To show complete causality or equivalently, that $(E,A)$ has index $\nu\leq 1$, we consider the lower diagonal block of \eqref{eq:gen_lyap} and get
\[
-X_{22}+N^HX_{22}N\geq 0.
\]
%Denote the index of $(E,A)$ by $\nu\in\mathbb{N}$. 
Since $N^\nu=0$, then from \eqref{eq:gen_lyap} we get
\[
-(N^H)^{\nu-1}X_{22}N^{\nu-1}=-(N^H)^{\nu-1}X_{22}N^{\nu-1}+(N^H)^\nu X_{22} N^\nu\geq 0,
\]
which implies that
\[
-(N^H)^{\nu-2}X_{22}N^{\nu-2}\geq-(N^H)^{\nu-2}X_{22}N^{\nu-2}+(N^H)^{\nu-1}X_{22}N^{\nu-1}\geq 0.
\]
Repeating this argument inductively leads eventually to $-X_{22}\geq 0$. Since $X\geq 0$ we have $X_{22}\geq 0$ and hence $X_{22}=0$. If the index $\nu$ is larger than one, then $N\neq 0$ and hence the positivity condition in \eqref{eq:gen_lyap} implies that for $x=(0,Nx_2)\in\ima E\in\setminus\{0\}$ with $x_2\in\K^{n-r}$ we have 
\[
0=x_2N^HX_{22}Nx_2=x^HXx>0
\]
which is a contradiction. Thus, $N=0$ which implies that \eqref{dae_discr} has an index at most one and is therefore completely causal.
\end{proof}

In the following example we show that without the additional definiteness assumptions on the solution $X$ of \eqref{eq:gen_lyap} in Proposition~\ref{prop:lyapforcausalstable}, we can neither conclude stability nor causality.
\begin{example}\label{ex:noncausal}{\rm 
Consider $E=\begin{bmatrix}
0&1\\0&0
\end{bmatrix}$ and $A=\begin{bmatrix}
1&0\\0&1
\end{bmatrix}$ so that \eqref{discr_DAE} has index $\nu=2$ and is therefore not causal. On the other hand, $X=0$ is a positive semidefinite solution of \eqref{eq:gen_lyap}.  Furthermore, there exists no solution $X$ of \eqref{eq:gen_lyap} which is positive definite on $\ima E$ which is spanned by $e_1= [1,0]^\top $. This follows since then $e_1^H X e_1>0$, but \eqref{eq:gen_lyap} implies 
\[
0\leq -e_1^HA^HXAe_1+e_1^HE^HXEe_1=-e_1^HXe_1<0
\]
which is a contradiction. 

If we consider $E=\begin{bmatrix}
1&0\\0&1
\end{bmatrix}$ and $A=\begin{bmatrix}
1&1\\0&1
\end{bmatrix}$, then the pair has a Jordan block of size $2\times 2$ at the eigenvalue $\lambda=1$. By Proposition~\ref{prop:stable}, the system \eqref{dae_discr} is not stable and there exists no positive definite solution $X$ of the Lyapunov inequality \eqref{eq:gen_lyap}. On the other hand \eqref{eq:gen_lyap} has the trivial positive semidefinite solution $X=0$.
}
\end{example}

\section{Proofs from Section~\ref{sec:Dissipation}}
\subsection{Proof of Proposition~\ref{prop:KYPthenPa}}
\label{app:KYPthenPa}

\begin{proof}[]
If the system fufills (d-iKYP) for some $X=X^H\geq 0$, then there exists $V(x)=x^HXx$ such that for all $x_k\in\K^n$ and $u_k\in\K^m$ we have 
\begin{align}
\label{Pa_to_imp_KYP}
0&\leq \begin{bmatrix}
x_k\\ u_k
\end{bmatrix}^H\begin{bmatrix}
-A^HXA+E^HXE& C^H-A^HXB\\ C-B^HXA& D+D^H-B^HXB
\end{bmatrix}\begin{bmatrix}
x_k\\ u_k
\end{bmatrix}\\ \nonumber
&=\begin{bmatrix}
x_k\\ u_k
\end{bmatrix}^H\begin{bmatrix}
-A^HXA+E^HXE&-A^HXB\\ -B^HXA&-B^HXB
\end{bmatrix}\begin{bmatrix}
x_k\\ u_k
\end{bmatrix}+2\Re (y_k^Hu_k) \\ \nonumber
&=-(Ax_k+Bu_k)^HX(Ax_k+Bu_k)+(Ex_k)^HXEx_k +2\Re (y_k^Hu_k) \\
&=-V(Ex_{k+1})+V(Ex_k) +2\Re ( y_k^Hu_k). \nonumber
\end{align}
If the system fulfills (d-sKYP), then there exists $V(x)=x^HXx$ such that for all $x_k\in\K^n$ and $u_k\in\K^m$ we have
\begin{align}\label{skypproof}
0&\leq \begin{bmatrix}
x_k\\ u_k
\end{bmatrix}^H\begin{bmatrix}
-A^HXA+E^HXE-C^HC& -C^HD-A^HXB\\ -D^HC-B^HXA& I-D^HD-B^HXB
\end{bmatrix}\begin{bmatrix}
x_k\\ u_k
\end{bmatrix}\\ \nonumber
&=\begin{bmatrix}
x_k\\ u_k
\end{bmatrix}^H\begin{bmatrix}
-A^HXA+E^HXE&-A^HXB\\ -B^HXA&-B^HXB
\end{bmatrix}\begin{bmatrix}
x_k\\ u_k
\end{bmatrix}+u_k^Hu_k-y_k^Hy_k \\ \nonumber
&=-(Ax_k+Bu_k)^HX(Ax_k+Bu_k)+(Ex_k)^HXEx_k +u_k^Hu_k-y_k^Hy_k  \\
&=-V(Ex_{k+1})+V(Ex_k) +u_k^Hu_k-y_k^Hy_k. \nonumber
\end{align}
\end{proof}

\subsection{Proof of Corollary~\ref{cor: pakyp}}
\label{app:pakyp}

\begin{proof}[]
The property (d-iPa) implies \eqref{Pa_to_imp_KYP} for solutions $(x_k,u_k)$. For $k=0$ we obtain from \eqref{eq:indexone_ODE} that 
\[
\begin{bmatrix} x_0^1\\x_0^2\\u_0\end{bmatrix}=\begin{bmatrix}\Sigma_E^{-\tfrac{1}{2}}\widehat x_0\\A_{22}^{-1}(-A_{21}x_0^1-B_2u_0)\\u_0\end{bmatrix}=\underbrace{\begin{bmatrix} \Sigma_E^{-\tfrac{1}{2}}&0\\-A_{22}^{-1}A_{21}\Sigma_E^{-\tfrac{1}{2}}&-A_{22}^{-1}B_2\\0&I_m\end{bmatrix}}_{=:S}\begin{bmatrix}
\widehat x_0\\ u_0
\end{bmatrix}
\]
for some $\widehat x_0\in\dC^r$ and some $u_0\in\dC^m$. This and \eqref{Pa_to_imp_KYP} imply
\begin{align*}
&0\leq S^H \begin{bmatrix}
-A^H XA+E^HXE&C^H-A^HXB\\C-B^HXA&D+D^H-B^HXB
\end{bmatrix}S\\
&
=S^H \begin{bmatrix}
-A^H X\begin{bmatrix}\Sigma_E^{\tfrac{1}{2}}\Ac\\0\end{bmatrix}+E^HX\begin{bmatrix}\Sigma_E^{\tfrac{1}{2}}\\0\end{bmatrix}&\begin{bmatrix} C_1^H \\ C_2^H
\end{bmatrix}-A^HX\begin{bmatrix} \Sigma_E^{\tfrac{1}{2}}\Bc\\0\end{bmatrix} \\ 
\Cc-B^HX\begin{bmatrix}\Sigma_E^{\tfrac{1}{2}}\Ac\\0\end{bmatrix}
&\Dc+D^H-B^HX\begin{bmatrix} \Sigma_E^{\tfrac{1}{2}}\Bc\\0\end{bmatrix}
\end{bmatrix} \\
&=\begin{bmatrix}
-\begin{smallbmatrix} \Ac^H\Sigma_E^{\tfrac{1}{2}}&0\end{smallbmatrix}X\begin{smallbmatrix}\Sigma_E^{\tfrac{1}{2}}\Ac\\0\end{smallbmatrix}+\begin{smallbmatrix} \Sigma_E^{\tfrac{1}{2}}&0 \end{smallbmatrix}X\begin{smallbmatrix}\Sigma_E^{\tfrac{1}{2}}\\0\end{smallbmatrix}& \Cc^H - \begin{smallbmatrix} \Ac^H\Sigma_E^{\tfrac{1}{2}}&0\end{smallbmatrix}X\Sigma_E^{\tfrac{1}{2}}\begin{smallbmatrix}\Bc\\0\end{smallbmatrix} \\
\Cc-\begin{smallbmatrix}\Bc^H\Sigma_E^{\tfrac{1}{2}}&0\end{smallbmatrix}X\begin{smallbmatrix}\Sigma_E^{\tfrac{1}{2}}\Ac\\0\end{smallbmatrix}
&\Dc+\Dc^H-\begin{smallbmatrix} \Sigma_E^{\tfrac{1}{2}}\Bc\\0\end{smallbmatrix}^HX\begin{smallbmatrix} \Sigma_E^{\tfrac{1}{2}}\Bc\\0\end{smallbmatrix}
\end{bmatrix}\\
&=\begin{bmatrix}
- \Ac^H\Sigma_E^{\tfrac{1}{2}}X_{11}\Sigma_E^{\tfrac{1}{2}}\Ac+\Sigma_E^{\tfrac{1}{2}}X_{11}\Sigma_E^{\tfrac{1}{2}}& \Cc^H - \Ac^H\Sigma_E^{\tfrac{1}{2}}X_{11}\Sigma_E^{\tfrac{1}{2}}\Bc \\
\Cc-\Bc^H\Sigma_E^{\tfrac{1}{2}}X_{11}\Sigma_E^{\tfrac{1}{2}}\Ac
&\Dc+\Dc^H-\Bc^H \Sigma_E^{\tfrac{1}{2}}X_{11}\Sigma_E^{\tfrac{1}{2}}\Bc
\end{bmatrix},
\end{align*}
where we used in the last step the decomposition  $X=\begin{bmatrix}X_{11}&X_{12}\\ X_{21}& X_{22}\end{bmatrix}$. Then the condition $X=X^H\geq 0$ implies that $\mathcal{X}:=\Sigma_E^{\tfrac{1}{2}}X_{11}\Sigma_E^{\tfrac{1}{2}}\geq 0$. Conversely, if (d-iKYP) holds for the system $(I_n,\Ac,\Bc,\Cc,\Dc)$ for some $\mathcal{X}\geq 0$ then we can define $X:=\begin{bmatrix}
\mathcal{X}&0\\0&0
\end{bmatrix}\geq 0$ which can be used as a storage function implying that (d-iPa) holds. 

The proof for the scattering passive case is analogous.
\end{proof}

\subsection{Proof of Lemma~\ref{lem:KYP_without_supply}}
\label{app:KYP_without_supply}

\begin{proof}[]  The identity holds from the following calculations 
\begin{align*}
 & \begin{bmatrix}
    (zE-A)^{-1}B\\ I_m
    \end{bmatrix}^H\begin{bmatrix}
    -A^HXA+E^HXE& -A^HXB\\ -B^HXA & -B^HXB
    \end{bmatrix}\begin{bmatrix}
    (zE-A)^{-1}B\\ I_m
    \end{bmatrix}\\ &=\begin{bmatrix}
    (zE-A)^{-1}B\\ I_m
    \end{bmatrix}^H\begin{bmatrix}
    -A^HX(A-zE+zE)+E^HXE& -A^HXB\\ -B^HX(A-zE+zE) & -B^HXB
    \end{bmatrix}\begin{bmatrix}
    (zE-A)^{-1}B\\ I_m
    \end{bmatrix}\\ 
    &=\begin{bmatrix}
    (zE-A)^{-1}B\\ I_m
    \end{bmatrix}^H\begin{bmatrix}
    -zA^HXE(zE-A)^{-1}B+E^HXE(zE-A)^{-1}B\\-zB^HXE(zE-A)^{-1}B
    \end{bmatrix}\\
    &=\begin{smallbmatrix}
    (zE-A)^{-1}B\\ I_m
    \end{smallbmatrix}^H\begin{smallbmatrix}
    -z(A+zE-zE)^HXE(zE-A)^{-1}B+E^HXE(zE-A)^{-1}B\\-zB^HXE(zE-A)^{-1}B
    \end{smallbmatrix}\\
    &=B^H(zE-A)^{-H}(-z(A+zE-zE)^HXE(zE-A)^{-1}B+E^HXE(zE-A)^{-1}B)\\&~~~+(-zB^HXE(zE-A)^{-1}B)\\
    &=(1-\vert z\vert^2)B^H(zE-A)^{-H}E^HXE(zE-A)^{-1}B.
\end{align*}
\end{proof}

\subsection{Proof of Proposition~\ref{prop: brpa}}
\label{app:brpa}

 \begin{proof}[]
Impedance passivity (d-iPa) implies that there exists $X\geq 0$ satisfying the inequality
\begin{align}
\label{eq:kyp_in_proof}
 0\leq \begin{bmatrix}
    x_k\\ u_k
    \end{bmatrix}^H\begin{bmatrix}
    -A^HXA+E^HXE& C^H-A^HXB\\ C-B^HXA & D+D^H-B^HXB
    \end{bmatrix}\begin{bmatrix}
    x_k\\ u_k
    \end{bmatrix}
\end{align}
for all $x_k\in\K^n$, $u_k\in\K^m$, which solve \eqref{discr_DAE} for all $k\geq 0$.  Furthermore, \cite[Proposition 2.11]{BanV19} and Lemma~\ref{lem:KYP_without_supply} imply that
\begin{align*}
 0&\leq \begin{bmatrix}
    (zE-A)^{-1}B\\ I_m
    \end{bmatrix}^H\begin{bmatrix}
    -A^HXA+E^HXE& C^H-A^HXB\\ C-B^HXA & D+D^H-B^HXB
    \end{bmatrix}\begin{bmatrix}
    (zE-A)^{-1}B\\ I_m
    \end{bmatrix}
    \\&=(1-\vert z\vert^2)B^H(zE-A)^{-H}E^HXE(zE-A)^{-1}B\\&~~~~+\begin{bmatrix}
    (zE-A)^{-1}B\\ I_m
    \end{bmatrix}^H\begin{bmatrix}
    0& C^H\\ C& D+D^H
    \end{bmatrix}\begin{bmatrix}
    (zE-A)^{-1}B\\ I_m
    \end{bmatrix}
    \\&=(1-\vert z\vert^2)B^H(zE-A)^{-H}E^HXE(zE-A)^{-1}B\\&~~~~+C(zE-A)^{-1}B+(C(zE-A)^{-1}B)^H+D+D^H.
\end{align*}
Since $X\geq0$, we have for all $\vert z\vert>1$ with $z\in\dC\setminus\sigma(E,A)$ that $\Tc(z)+\Tc(z)^H\geq 0$ holds. To conclude that the transfer function is positive real it remains to show that it has no poles $z$ satisfying $\vert z\vert>1$. The set of poles of $\Tc(z)=C(zE-A)^{-1}B+D$ is a subset of $\sigma(E,A)$. Hence, if all elements $z\in\sigma(E,A)$ have modulus at most one, then the transfer function is positive real. Therefore suppose that there exists $z\in\sigma(E,A)$ with $\vert z\vert>1$. Then $zEx=Ax$ for some $x\in\dC^n$, $x\neq 0$, and for $(x_0,u_0)=(x,0)$ the KYP inequality 
(d-iKYP) (resp. (d-sKYP)) leads to
\begin{align}
\label{eq:upperblock}
0\leq -x^HA^HXAx+x^HE^HXEx=(Ex)^H((1-\vert z\vert^2)X)Ex.
\end{align}
The regularity of $(E,A)$ and the fact that $z$ is finite, implies $Ex\neq 0$. 
Therefore, \eqref{eq:upperblock} holds with equality.
Hence $Ex\in\ker X$ and also $Ax=zEx\in\ker X$. Considering the KYP inequality \eqref{eq:kyp_in_proof} applied to elements of the form $[\alpha x^\top,\alpha u^\top]^\top$ for some  $\alpha>0$ and $u=-Cx$ leads to 
\[
0\leq -\alpha 2\re (Cx)^HCx+\alpha^2(Cx)^H(D+D^H-B^HXB)Cx.
\]
Thus, for $\alpha>0$ sufficiently small 
for  sufficiently small,  we conclude that $Cx=0$ and therefore the eigenvalue $z$ is not a pole of the transfer function.

An analogous proof for (d-sPa) follows from 

\[
\begin{bmatrix}
    (zE-A)^{-1}B\\ I_m
    \end{bmatrix}^H\begin{bmatrix}
   -C^HC& -C^HD\\ -D^HC & I_n-D^HD
    \end{bmatrix}\begin{bmatrix}
    (zE-A)^{-1}B\\ I_m
    \end{bmatrix}=I_m-\mathcal{T}(z)^H\mathcal{T}(z),
\]
which implies 
that $I_m-\mathcal{T}(z)^H\mathcal{T}(z)\geq 0$ for all $\vert z\vert> 1$ with $z\in\dC\setminus\sigma(E,A)$. 
\end{proof}

\subsection{Proof of Lemma~\ref{lem:TF}}

\label{app:TF}

\begin{proof}[]
Assertion (a) follows from standard calculations using the Schur complement and is therefore left to the reader. Here one may assume without loss of generality that the system is already in the semiexplicit form \eqref{semiexplicit}. Then the transfer functions are related in the following way
\begin{align}\label{dtoE}
\Tc(z)&=\begin{bmatrix}C_1 &C_2\end{bmatrix}\begin{bmatrix}
z\Sigma_E-A_{11}&-A_{12}\\ -A_{21}&-A_{22}
\end{bmatrix}^{-1}\begin{bmatrix} B_1\\ B_2\end{bmatrix} + D\nonumber\\
&=\mathcal{C}(z\Sigma_E-\mathcal{A})^{-1}\mathcal{B}+\mathcal{D}\nonumber \\
&=\mathcal{C}\Sigma_E^{-\tfrac{1}{2}}(zI_r-\mathcal{A})^{-1}\Sigma_E^{-\tfrac{1}{2}}\mathcal{B}+\mathcal{D}.
\end{align}
For the proof of (b) we may again assume that the system is in the semiexplicit form \eqref{semiexplicit}, because the transformations to this form neither changes the controllability nor the observability properties of the system.  Since $A_{22}$ is invertible, we have that  $\begin{bmatrix}\lambda E-A & B \end{bmatrix}$ has full rank if and only if  $\begin{bmatrix}\lambda \Sigma_E-\Ac &\Bc\end{bmatrix}$ has full rank. Similarly, it follows that  
$\begin{smallbmatrix}
\lambda E-A\\ C
\end{smallbmatrix}$ has full rank if and only if $\begin{smallbmatrix}
\lambda \Sigma_E-\Ac\\ \Cc 
\end{smallbmatrix}$ has full rank. Furthermore, by applying invertible transformations we get the two rank conditions 
\begin{align*}
\rk \begin{bmatrix}
\lambda I_n-\Sigma_E^{-\tfrac{1}{2}}\Ac\Sigma_E^{-\tfrac{1}{2}}\\ \Cc\Sigma_E^{-\tfrac{1}{2}} 
\end{bmatrix}&=\rk\begin{bmatrix}
\lambda \Sigma_E-\Ac\\ \Cc 
\end{bmatrix}=n,\\ \rk \begin{bmatrix}
\lambda I_n-\Sigma_E^{-\tfrac{1}{2}}\Ac\Sigma_E^{-\tfrac{1}{2}} & \Sigma_E^{-\tfrac{1}{2}}\Bc 
\end{bmatrix}&=\rk\begin{bmatrix}
\lambda \Sigma_E-\Ac &\Bc 
\end{bmatrix}=n.
\end{align*}
\end{proof}

\end{appendix}

\end{document}